\documentclass{article}

\usepackage{amscd,amsmath,amssymb,amsthm,xcolor,mathrsfs,dsfont,bm}
\newcommand{\eps}{\varepsilon}

\newcommand{\R}{\mathbb R}
\newcommand{\N}{\mathbb N}

\newcommand{\then}{\Longrightarrow}
\newcommand{\J}{{\cal J}}

\newcommand{\M}{{\cal M}}
\newcommand{\enne}{{\cal N}}
\newcommand{\er}{{\cal R}}

\newcommand{\T}{{\cal T}}

\newcommand{\bu}{\mathbf{u}}

\newcommand{\bv}{\mathbf{v}}

\newcommand\meas{{\rm meas}}
\newcommand{\ws}{\widetilde{s}}

\DeclareMathOperator{\codim}{codim}
\DeclareMathOperator*{\esssup}{ess\; sup}

\newtheorem{corollary}{Corollary}[section]
\newtheorem{theorem}[corollary]{Theorem}
\newtheorem{lemma}[corollary]{Lemma}
\newtheorem{proposition}[corollary]{Proposition}

\theoremstyle{definition}
\newtheorem{definition}[corollary]{Definition}
\newtheorem{remark}[corollary]{Remark}
\newtheorem{example}[corollary]{Example}

\numberwithin{equation}{section}
\begin{document}

\title{{\bf Nontrivial solutions for a class of\\
 gradient--type quasilinear elliptic systems}
\footnote{The research that led to the present paper was partially supported 
by MIUR--PRIN project ``Qualitative and quantitative aspects of nonlinear PDEs'' 
(2017JPCAPN\underline{\ }005),
{\sl Fondi di Ricerca di Ateneo} 2017/18 ``Problemi differenziali non lineari''  }}

\author{Anna Maria Candela and Caterina Sportelli\\
{\small Dipartimento di Matematica}\\
{\small Universit\`a degli Studi di Bari Aldo Moro} \\
{\small Via E. Orabona 4, 70125 Bari, Italy}\\
{\small \it annamaria.candela@uniba.it}\\
{\small \it caterina.sportelli@uniba.it}}
\date{}

\maketitle

\begin{abstract}%
The aim of this paper is investigating the existence of weak bounded solutions 
of the gradient--type quasilinear elliptic system
\[
(P)\qquad \left\{
\begin{array}{ll}
- {\rm div} ( a_i(x, u_i, \nabla u_i) ) 
+ A_{i, t} (x, u_i, \nabla u_i) = G_i(x, \bu) &\hbox{ in $\Omega$}\\ [5pt] 
\quad\qquad\qquad\qquad\qquad \mbox{ for }\;  i\in\{1,\dots,m\},\\[5pt]
\bu = 0 &\hbox{ on $\partial\Omega$,}
\end{array}
\right.
\]
with $m\geq 2$ and $\bu=(u_1,\dots, u_{m})$, where 
$\Omega\subset\R^N$ is an open bounded domain and 
some functions $A_i:\Omega\times\R\times\R^N\rightarrow\R$, $i\in\{1,\dots,m\}$,
and $G:\Omega\times\R^m\rightarrow\R$ exist 
such that $a_i(x,t,\xi) = \nabla_{\xi} A_i(x,t,\xi)$, 
$A_{i, t} (x,t,\xi) = \frac{\partial A_i}{\partial t} (x,t,\xi)$ 
and $G_{i}(x,\bu) = \frac{\partial G}{\partial u_i}(x,\bu)$.

We prove that, under suitable hypotheses, the functional $\J$ related to problem $(P)$ 
is $\mathcal{C}^1$ on a ``good'' Banach space $X$ and satisfies the weak Cerami--Palais--Smale condition. 
Then, generalized versions of the Mountain Pass Theorems allow us to prove 
the existence of at least one critical point and, if $\J$ is even, of infinitely many ones, too.
\end{abstract}

\noindent
{\it \footnotesize 2020 Mathematics Subject Classification}. {\scriptsize 35J50, 35J92, 47J30, 58E05}.\\
{\it \footnotesize Key words}. {\scriptsize Gradient--type quasilinear elliptic system, $p$--Laplacian type operator,
subcritical growth, weak Cerami--Palais--Smale condition, Ambrosetti--Rabinowitz condition, Mountain Pass theorem, 
even functional, pseudo--eigenvalue}.


\section{Introduction} \label{secintroduction} 

In this paper we look for weak bounded solutions of the following class of 
gradient--type quasilinear elliptic systems
\begin{equation}   \label{euler}
\left\{
\begin{array}{ll}
- {\rm div} ( a_i(x, u_i, \nabla u_i) ) 
+ A_{i, t} (x, u_i, \nabla u_i) = G_i(x, \bu) &\hbox{ in $\Omega$}\\ [5pt] 
\quad\qquad\qquad\qquad\qquad \mbox{ for }\;  i\in\{1,\dots,m\},\\[5pt]
\bu = 0 &\hbox{ on $\partial\Omega$,}
\end{array}
\right.
\end{equation}
with $m\geq 2$ and $\bu=(u_1,\dots, u_{m})$, where 
$\Omega\subset\R^N$ is an open bounded domain and
some functions $A_i:\Omega\times\R\times\R^N\rightarrow\R$, $i\in\{1,\dots,m\}$,
and $G:\Omega\times\R^m\rightarrow\R$ exist such that
\begin{equation} \label{Ata}
A_{i, t} (x, t, \xi) =\displaystyle\frac{\partial A_i}{\partial t}(x, t, \xi), 
\; a_i(x, t, \xi) = \left(\frac{\partial A_i}{\partial \xi_{ 1}}(x, t, \xi), \dots, 
\frac{\partial A_i}{\partial \xi_{N}}(x, t, \xi)\right)
\end{equation}
if $1\le i\le m$ and $\nabla_{\bu} G(x,\bu) = (G_1(x,\bu),\dots, G_m(x,\bu)$, i.e.,
\begin{equation}   \label{GuGv}
G_{i}(x, \bu) = \frac{\partial G}{\partial u_i}(x, \bu)\qquad \hbox{if}\ 1 \le i\le m.
\end{equation}
A special model of system \eqref{euler} is obtained if 
$A_{i}(x,t,\xi) = \frac{1}{p_i}\bar{A}_{i} (x,t)\vert\xi\vert^{p_i}$, $p_i > 1$, 
so \eqref{euler} reduces to problem 
\[
\left\{
\begin{array}{ll}
- {\rm div} (\bar{A}_i(x, u_i)\vert\nabla u_i\vert^{p_i -2} \nabla u_i) 
+ \frac{1}{p_i}\bar{A}_{i,t} (x, u_i)\vert\nabla u_i\vert^{p_i} = G_{i}(x, \bu) 
&\hbox{ in $\Omega$}\\ [10pt]
\quad\qquad\qquad\qquad\qquad\qquad\qquad \mbox{ for }\;  i\in\{1,\dots,m\},\\[5pt]
\bu= 0 &\hbox{ on $\partial\Omega$,}
\end{array}
\right.
\]
with $\bar{A}_{i,t} (x,t) = \frac{\partial \bar{A}_i}{\partial t}(x,t)$, 
which has been studied in \cite{CSS} if $m=2$, and generalizes 
the classical gradient-type $(p_1,\dots,p_m)$--Laplacian system 
\[
\left\{
\begin{array}{ll}
- \Delta_{p_i} u_i = G_{i}(x, \bu) &\hbox{ in $\Omega$, for $i\in\{1,\dots,m\}$,}\\ [10pt]
\bu= 0 &\hbox{ on $\partial\Omega$,}
\end{array}
\right.
\]
which has been widely analized in the past (see, e.g., \cite{dF1,dF2,PAO,VT}). 

Many variants of system \eqref{euler} have been studied 
by using several theories such as the fixed point index, 
the cohomological index, the sub--super solutions methods, the bifurcation theory 
and also some non--variational techniques 
(see, e.g., \cite{AG,dF1,CMPP,dF2,dFR,FT,PeSq,VT} and references therein).

On the contrary, here we use the variational approach introduced in \cite{CP1,CP2}, 
and already applied to systems in \cite{CSS}, so that, under suitable hypotheses, 
finding solutions of problem \eqref{euler}
turns into searching critical points of the functional 
\begin{equation}      \label{functional}
\J(\bu)\ =\ \sum_{i =1}^{m}\int_{\Omega} A_i(x, u_i, \nabla u_i) dx\ -\ \int_{\Omega} G(x, \bu) dx
\end{equation}
in a suitable Banach space $X$ obtained as product of the intersection spaces between
Sobolev spaces and $L^\infty(\Omega)$ (for more details, see Section \ref{variational}).

Differently from the problem in \cite{CSS},
here we deal with a more general system. In spite of it,
a regularity result on $\J$ in $X$ can be proved under basic assumptions on 
$G(x, \bu)$ and suitable growth conditions on the $\mathcal{C}^1$--Carath\'eodory function 
$A_i(x, t, \xi)$ and its partial derivatives, $i \in \{1,\dots,m\}$, 
but pointing out that 
it has to growth with power $p_i>1$ with respect to $\xi$
(see Proposition \ref{smooth1}).
Since, in general, $\J$ satisfies neither the Palais--Smale condition nor its 
classical Cerami variant (see \cite[Example 4.3]{CP2017}),
following the lead of ideas developed in \cite{CP2}, which
exploit the interaction between two different norms on $X$,
we inquire whether $\J$ verifies the weak Cerami--Palais--Smale condition (see Definition \ref{wCPS})
so to apply the abstract theorems in \cite{CP3}.
To this aim, by considering not the interplay just between 
$A_i(x, t, \xi)$ and the partial derivative $A_{i, t}(x, t, \xi)$ as in \cite{CSS}, 
but the interaction among them and $a_i(x, t, \xi)$ 
(see the hypotheses at the beginning of Section \ref{sec_wcps}),
 together with an Ambrosetti--Rabinowitz type condition and suitable subcritical growth 
assumptions on $G(x,\bu)$, we have that the weak Cerami--Palais--Smale condition holds
(see Proposition \ref{PropwCPS}).
Then, suitable requirements on the behavior of $G(x, \bu)$ 
in a neighborhood of the origin, respectively at infinity, 
allow us to state an existence result,
respectively a multiplicity one if $\J$ is even
by means of a ``good'' decomposition of the Sobolev spaces $W^{1,p_i}_0(\Omega)$  
as given in \cite[Section 5]{CP2}. 

Anyway, in order to not weigh this introduction down
with too many details, we prefer to specify each
hypothesis when required and to state 
our main results at the beginning of Section \ref{sec_main}
(see Theorems \ref{ThExist} and \ref{ThMolt}).
Note that, here, we look for bounded solutions of problem
\eqref{euler}, so we introduce some subcritical growth 
hypotheses on $G(x,\bu)$ which are stronger than the usual ones 
(compare \eqref{crit_expi} with \cite[Theorem 3]{dF1}).

This paper is organized as follows.
In Section \ref{abstractsection} we introduce the abstract tools and, in particular,
the weak Cerami--Palais--Smale condition
and some related existence and multiplicity results which generalize classical Mountain Pass Theorems.
In Section \ref{variational}, we introduce the variational setting and 
give the first assumptions in order to prove the variational principle required by problem \eqref{euler}. 
Then, in Section \ref{sec_wcps} we prove that functional $\J$ 
satisfies the weak Cerami--Palais--Smale condition and, finally, in Section \ref{sec_main}, 
our main results are stated and proved.


\section{Abstract tools}
\label{abstractsection}

We denote $\N = \{1, 2, \dots\}$ and, throughout this section, we assume that:
\begin{itemize}
\item $(X, \|\cdot\|_X)$ is a Banach space with dual 
$(X',\|\cdot\|_{X'})$;
\item $(W,\|\cdot\|_W)$ is a Banach space such that
$X \hookrightarrow W$ continuously, i.e. $X \subset W$ and a constant $\sigma_0 > 0$ exists
such that
\[
\|y\|_W \ \le \ \sigma_0\ \|y\|_X\qquad \hbox{for all $y \in X$;}
\]
\item $J : {\mathcal D} \subset W \to \R$ and $J \in C^1(X,\R)$ with $X \subset {\mathcal D}$.
\end{itemize}

Anyway, in order to avoid any
ambiguity and simplify, when possible, the notation, 
from now on by $X$ we denote the space equipped with
its given norm $\|\cdot\|_X$ while, if the norm $\Vert\cdot\Vert_{W}$ is involved,
we write it explicitly.

For simplicity, taking $\beta \in \R$, we say that a sequence
$(y_n)_n\subset X$ is a {\sl Cerami--Palais--Smale sequence at level $\beta$},
briefly {\sl $(CPS)_\beta$--sequence}, if
\[
\lim_{n \to +\infty}J(y_n) = \beta\quad\mbox{and}\quad 
\lim_{n \to +\infty}\|dJ\left(y_n\right)\|_{X'} (1 + \|y_n\|_X) = 0.
\]

As $(CPS)_\beta$--sequences may exist which are unbounded in $\|\cdot\|_X$
but converge with respect to $\|\cdot\|_W$ (in our setting 
with $m=1$, see \cite[Example 4.3]{CP2017}),
we have to weaken the classical Cerami--Palais--Smale 
condition in a suitable way according to the ideas already developed in 
previous papers (see, e.g., \cite{CP1,CP2,CP3}).  

\begin{definition} \label{wCPS}
The functional $J$ satisfies the
{\slshape weak Cerami--Palais--Smale 
condition at level $\beta$} ($\beta \in \R$), 
briefly {\sl $(wCPS)_\beta$ condition}, if for every $(CPS)_\beta$--sequence $(y_n)_n$,
a point $y \in X$ exists, such that 
\begin{itemize}
\item[{\sl (i)}] $\displaystyle 
\lim_{n \to+\infty} \|y_n - y\|_W = 0\quad$ (up to subsequences),
\item[{\sl (ii)}] $J(y) = \beta$, $\; dJ(y) = 0$.
\end{itemize}
If $J$ satisfies the $(wCPS)_\beta$ condition at each level $\beta \in I$, $I$ real interval, 
we say that $J$ satisfies the $(wCPS)$ condition in $I$.
\end{definition}

Since the $(wCPS)_\beta$ condition allows one to prove 
a Deformation Lemma (see \cite[Lemma 2.3]{CP3}), then
 the following generalization of the Mountain Pass Theorem can be stated
(see \cite[Theorem 1.7]{CP3} and 
compare it with the classical statement \cite[Theorem 2.2]{Ra1}).

\begin{theorem}
\label{mountainpass}
Let $J\in C^1(X,\R)$ be such that $J(0) = 0$
and the $(wCPS)$ condition holds in $\R_+$.\\
Moreover, assume that there exist some constants $R_0$, $\varrho_0 > 0$, and  
 a point $e \in X$ such that 
\begin{itemize}
\item[$(i)$] $\; y \in X, \quad \|y\|_W = R_0\qquad \then\qquad J(y) \geq \varrho_0$;
\item[$(ii)$] $\; \| e\|_W > R_0\qquad\hbox{and}\qquad J(e) < \varrho_0$.
\end{itemize}
Then, $J$ has a Mountain Pass critical point $y \in X$ such that $J(y) \geq \varrho_0$.
\end{theorem}

Furthermore, with the stronger assumption that $J$ is even, also 
the symmetric Mountain Pass Theorem can be generalized as follows
(see \cite[Theorem 2.4]{CPS_CCM} or also \cite[Theorem 1.8]{CP3} 
and compare it with \cite[Theorem 9.12]{Ra1} and 
\cite[Theorem 2.4]{BBF}).

\begin{theorem}
\label{abstract}
Let $J\in C^1(X,\R)$ be an even functional such that $J(0) = 0$ and
the $(wCPS)$ condition holds in $\R_+$.
Moreover, assume that $\varrho > 0$ exists so that:
\begin{itemize}
\item[$({\mathcal H}_{\varrho})$]
three closed subsets $V_\varrho$, $Y_\varrho$ and $\M_\varrho$ of $X$ and a constant
$R_\varrho > 0$ exist which satisfy the following conditions:
\begin{itemize}
\item[$(i)$] $V_\varrho$ and $Y_\varrho$ are subspaces of $X$ such that
\[
V_\varrho + Y_\varrho = X,\qquad \codim Y_\varrho\ <\ \dim V_\varrho\ <\ +\infty;
\]
\item[$(ii)$] $\M_\varrho = \partial \enne$, where $\enne \subset X$ is a neighborhood of the origin
which is symmetric and bounded with respect to $\|\cdot\|_W$; 
\item[$(iii)$]  $\ y \in \M_\varrho \cap Y_\varrho\qquad \then\qquad J(y) \ge \varrho$;
\item[$(iv)$] $\ y \in V_\varrho, \quad \|y\|_X \ge R_\varrho \qquad \then\qquad J(y) \le 0$.
\end{itemize}
\end{itemize}
Then, if we put 
\[
\beta_\varrho\ =\ \inf_{\gamma \in \Gamma_\varrho} \sup_{y\in V_\varrho} J(\gamma(y)),
\]
with 
\[
\begin{split}
\Gamma_\varrho = \{\gamma : X \to X: \; &\gamma\ \hbox{odd homeomeorphism such that}\\
&\gamma(y) = y \ \hbox{if $y \in V_\varrho$ with $\|y\|_X \ge R_\varrho$}\},
\end{split}
\]
functional $J$ possesses at least a pair of symmetric critical points in $X$ 
with corresponding critical level $\beta_\varrho$ which belongs to $[\varrho,\varrho_1]$,
where $\varrho_1 \ge \displaystyle \sup_{y \in V_\varrho}J(y) > \varrho$.
\end{theorem}

If we can apply infinitely many times Theorem \ref{abstract}, then the following multiplicity abstract 
result can be stated.
\begin{corollary}
\label{multiple}
Let $J \in C^1(X,\R)$ be an even functional such that $J(0) = 0$,
the $(wCPS)$ condition holds in $\R_+$ and a sequence $(\varrho_n)_n \subset\ ]0,+\infty[$ exists
such that $\varrho_n \nearrow +\infty$ and  
assumption $({\mathcal H}_{\varrho_n})$ holds for all $n \in \N$.\\
Then, functional $J$ possesses a sequence of critical points $(u_{k_n})_n \subset X$ such that
$J(u_{k_n}) \nearrow +\infty$ as $n \nearrow +\infty$.
\end{corollary}

\section{Variational setting and first properties}
\label{variational}

From now on, let $\Omega \subset \R^N$ be an open bounded domain, $N\ge 2$,
and $m \in \N$ such that $m\ge 2$,
so we denote by:
\begin{itemize}
\item $\bu =(u_1,\dots, u_m)$, $\bu_n = (u^n_1,\dots, u^n_m)$, ${\bf 0}=(0,\dots, 0)\in\R^m$;
\item $\{{\bf e}_j : 1\leq j\leq m\}$ the standard basis of the Euclidean space $\R^m$, 
i.e., ${\bf e}_j$ has components $e_i^j = \delta_i^j$;
\item $L^r(\Omega)=L^r(\Omega,\R)$, $1 \le r < +\infty$, the classical Lebes\-gue space with
norm $|u|_r = \left(\int_\Omega|u|^r dx\right)^{1/r}$;
\item $L^\infty(\Omega)=L^\infty(\Omega,\R)$ the space of Lebesgue--measurable 
essentially\\ bounded functions with norm $\displaystyle |u|_{\infty} = \esssup_{\Omega} |u|$;
\item $W^{1,p}_0(\Omega)=W^{1,p}_0(\Omega,\R)$ the classical Sobolev space equipped with
norm $\|u\|_{W_0^{1, p}} = |\nabla u|_{p}$ if $1 \le p < +\infty$;
\item $\meas(D)$ the usual Lebesgue measure of a measurable set $D$ in $\R^N$.
\end{itemize}

For simplicity, here and in the following we denote by $|\cdot|$  
the standard norm on any Euclidean space, as the dimension
of the considered vector is clear and no ambiguity occurs. 
Moreover, for short, we replace
\[
\sum_{i=1}^m\quad \hbox{with}\quad \sum_i
\qquad \hbox{and}\qquad 
\sum_{\substack{j =1\\ j\neq i}}^{m} \quad \hbox{with}\quad \sum_{j\ne i}.
\]

\begin{definition}
A function $f:\Omega\times\R^l\to\R$, $l \in \N$, is a $\mathcal{C}^{k}$--Carath\'eodory function, 
$k\in\N\cup\lbrace 0\rbrace$, if
\begin{itemize}
\item $f(\cdot, \omega) : x \in \Omega \mapsto f(x, \omega) \in \R$ is measurable for all $\omega \in \R^l$,
\item $f(x,\cdot) : \omega \in \R^l \mapsto f(x, \omega) \in \R$ is $\mathcal{C}^k$ for a.e. $x \in \Omega$.
\end{itemize}
\end{definition}

For each $i\in\{1,\dots, m\}$, let $A_i:(x,t, \xi) \in \Omega\times\R\times\R^N \mapsto A_i(x,t, \xi) \in \R$
be a given function such that the following conditions hold:
\begin{itemize}
\item[$(h_0)$]
$A_i(x,t,\xi)$ is a $\mathcal{C}^1$--Carath\'eodory function with partial derivatives
\[
A_{i,t}(x,t, \xi)\quad\mbox{ and }\quad a_i(x,t, \xi)
\]
as in \eqref{Ata};
\item[$(h_1)$] a power $p_i > 1$ and some positive continuous functions
 $\Phi_0^i$, $\phi_0^i$, $\Phi_1^i$, $\phi_1^i$, $\Phi_2^i$, $\phi_2^i:\R\to\R$ exist such that 
\begin{eqnarray}
|A_{i}(x, t, \xi)|&\leq&\Phi_0^i(t) +\phi_0^i(t)|\xi|^{p_i}, \label{conto}\\
| A_{i,t}(x, t, \xi)| &\leq&\Phi_1^i(t) +\phi_1^i(t)|\xi|^{p_i}, \label{conto1}\\
| a_i(x, t, \xi)| &\leq&\Phi_2^i(t) +\phi_2^i(t)|\xi|^{p_i-1}, \label{conto2}
\end{eqnarray}
for a.e. $x\in\Omega$ and for all $(t, \xi)\in\R\times\R^N$.
\end{itemize}

So, taking $p_i > 1$ as in $(h_1)$, we consider the related Sobolev space
\[
W_i = W_0^{1, p_i}(\Omega)\quad \hbox{with norm $\Vert \cdot\Vert_{W_i} = \Vert \cdot\Vert_{W_0^{1, p_i}}$.}
\]
From the Sobolev Embedding Theorem, 
$W_i$ is continuously embedded in $L^{r}(\Omega)$ for any $r\in[1, p_i^{\ast}]$ 
with $p_i^{\ast} = \frac{Np_i}{N-p_i}$ if $N>p_i$, 
or $r \in [1,+\infty[$ with $p^*_i = +\infty$ if $p_i \ge N$, i.e.,
for such an $r$ a positive constant 
$\tau_{i,r}$ exists such that
\begin{equation}   \label{Sobpi}
\vert u\vert_r \leq\tau_{i,r}\Vert u\Vert_{W_i} \quad \mbox{ for all } u\in W_i.
\end{equation}
For simplicity, we put
\begin{equation}     \label{0infinito}
\frac{1}{p^*_i} = 0 \quad \hbox{if} \quad p^*_i = +\infty.
\end{equation}

Now, assume that a function $G:(x,\bu) \in \Omega\times\R^m \mapsto G(x,\bu) \in \R$
exists such that
\begin{itemize}
\item[$(g_0)$]
$G(x,\bu)$ is a $\mathcal{C}^1$--Caratheodory function with partial derivatives
$G_i(x,\bu)$ as in \eqref{GuGv} such that
\[
G(\cdot,\textbf{0}) \in L^\infty(\Omega)
\]
and
\[
G_{i}(x, {\bf 0}) = 0 \quad \mbox{ for a.e. } x\in\Omega, \; \hbox{for each}\ i\in\{1,\dots, m\};
\]
\item[$(g_1)$] for every $i$, $j \in \{1,\dots,m\}$, $j\ne i$, some real numbers $q_i \geq 1$,
$s_{i, j} \geq 0$ and a constant $\sigma >0$ exist such that 
\begin{equation}      \label{gucr}
\vert G_{i}(x, \bu)\vert\leq \sigma\Bigg(1 +\vert u_i\vert^{q_i -1} +\sum_{j\neq i} \vert u_j\vert^{s_{i, j}}\Bigg)
\end{equation}
for a.e. $x\in\Omega$ and all $\bu \in \R^{m}$, with 
\begin{equation}    \label{crit_exp}
1\leq q_i < p_i^{\ast} 
\end{equation}
and  
\begin{equation}     \label{crit_expi}
0 \leq s_{i,j} < \frac{p_i}{N}\left(1-\frac{1}{p_i^{\ast}}\right) p_j^{\ast}.
\end{equation}
\end{itemize}

\begin{remark}
For a.e. $x\in\Omega$ and all $\bu\in\R^{m}$,
condition $(g_0)$ together with Mean Value Theorem implies that
$t \in ]0,1[$ exists such that
\[
|G(x,\bu)| \le \ |G(\cdot,\textbf{0})|_\infty + \sum_i |G_i(x,t\bu)| |u_i|,
\]
then from \eqref{gucr} it follows that
\begin{equation}    \label{Gsigma}
\vert G(x,\bu)\vert\ \leq\ \sigma_1 \sum_i \left(1 + \vert u_i\vert^{q_i} 
+\sum_{j\neq i} \vert u_i\vert \vert u_j\vert^{s_{i, j}}\right) 
\end{equation}
for a.e. $x\in\Omega$, all $\bu\in\R^m$, with $\sigma_1 > 0$ which depends on $\sigma$, $m$ and $|G(\cdot,\textbf{0})|_\infty$.
\end{remark}

In order to investigate the existence of weak solutions  
of the nonlinear problem \eqref{euler} as critical points of $\J$ defined as in \eqref{functional},
we have to introduce the ``right'' Banach space. To this aim,
the notation introduced for the abstract
setting in Section \ref{abstractsection} is referred to  
\begin{equation}     \label{Wdefn1}
W =W_1\times\dots\times W_m 
\end{equation}
with norm
\begin{equation}   \label{Wnorm}
\Vert\bu\Vert_W = \Vert(u_1,\dots,u_m)\Vert_W =\sum_i \Vert u_i\Vert_{W_i}
\quad \hbox{if $\bu\in W$,}
\end{equation}
while the Banach space $(X,\|\cdot\|_X)$ is defined as 
\begin{equation}     \label{Xdefn1}
X =X_1\times\dots\times X_m
\end{equation}
with norm
\[
\Vert\bu\Vert_X = \Vert(u_1,\dots,u_m)\Vert_X =\sum_i \Vert u_i\Vert_{X_i}
\quad \hbox{if $\bu\in X$,}
\]
where, for any $i\in\{1,\dots,m\}$, it is
\[
X_i:= W_i\cap L^{\infty}(\Omega)
\]
equipped with norm
\begin{equation}   \label{Xinorm}
\Vert u\Vert_{X_i} = \Vert u\Vert_{W_i} + \vert u\vert_{\infty} \quad\mbox{ if } u\in X_i.
\end{equation}
We note that, setting 
\[
L = L^{\infty}(\Omega)\times\dots \times L^{\infty}(\Omega)
\]
with norm
\[
\Vert \bu\Vert_{L}= \Vert(u_1,\dots,u_m)\Vert_L = \sum_{i} \vert u_i\vert_{\infty}
\quad \hbox{if $\bu\in L$,}
\]
we have that $X$ in \eqref{Xdefn1} can also be written as 
\[
X = W\cap L \qquad \hbox{with }\quad 
\Vert \bu\Vert_X = \Vert \bu\Vert_W + \Vert \bu\Vert_L. 
\]

For every $i\in\{1, \dots, m\}$, 
we have that $(W_i, \Vert\cdot\Vert_{W_i})$ is a reflexive Banach space
and, by definition, it is $X_i \hookrightarrow W_i$ and $X_i \hookrightarrow L^\infty(\Omega)$
with continuous embeddings.
Thus, also $(W, \Vert\cdot\Vert_W)$ is a reflexive Banach space
and, obvioulsy, $X \hookrightarrow W$ and $X \hookrightarrow L$ with continuous embeddings, too.

\begin{remark}
If $i\in\{1, \dots, m\}$ is such that $p_i > N$, then $X_i = W_i$, as $W_i \hookrightarrow L^\infty(\Omega)$.
So, in general, if an $i\in\{1, \dots, m\}$ exists such that $p_i \le N$ then $X \ne W$, 
but if for all $i\in\{1, \dots, m\}$ it is $p_i > N$, 
then $X = W$ and the classical Mountain Pass Theorems in \cite{AR} can be used, if required.
\end{remark}

If conditions $(h_0)$--$(h_1)$, $(g_0)$ and \eqref{gucr} hold,
by \eqref{Gsigma} and direct computations 
it follows that $\J(\bu)$ in \eqref{functional}
is well defined for all $\bu\in X$. Moreover,
taking any $\bu, {\bf v}\in X$, the G\^ateaux differential 
of functional $\J$ in $\bu$ along the direction ${\bf v}$ is well defined as
\begin{equation}     \label{diff}
\begin{split}
d\J(\bu)[\bv] =\ &\sum_{i}\left(\int_{\Omega} a_i(x, u_i, \nabla u_i)\cdot\nabla v_i \ dx\right.\\ 
&\qquad\left.+\int_{\Omega} A_{i, t}(x, u_i, \nabla u_i) v_i \ dx - \int_{\Omega} G_{i}(x, \bu) v_i \ dx\right).
\end{split}
\end{equation}
 
We note that, since $\bu, {\bf v}\in X$ imply that $\bu, {\bf v}\in L$, no 
critical growth upper bound on the powers $q_i$ and $s_{i,j}$ is required 
in order to have $d\J(\bu)[\bv] \in \R$.

For smplicity, for every $i\in\{1,\dots,m\}$ we introduce the $i$--th partial derivative of $\J$ in $\bu \in X$ as
\[
\frac{\partial\J}{\partial u_i}(\bu): v\in X_i\mapsto \frac{\partial\J}{\partial u_i}(\bu)[v] = d\J(\bu)[v {\bf e}_i]\in\R,
\]
where from \eqref{diff} it follows that
\begin{equation}     \label{dJw}
\begin{split}
\frac{\partial\J}{\partial u_i}(\bu) [v] \ &= \int_{\Omega} a_i(x, u_i, \nabla u_i)\cdot\nabla v\ dx 
+ \int_{\Omega} A_{i,t}(x, u_i, \nabla u_i) v\ dx\\
&\quad - \int_{\Omega} G_{i} (x, \bu) v\ dx.
\end{split}
\end{equation}

\begin{remark}      \label{RemNot}
Taking $\bu\in X$, since $d\J(\bu)\in X^{\prime}$, 
then 
\[
\frac{\partial\J}{\partial u_i}(\bu)\in X_i^{\prime}
\qquad \hbox{for all $i \in \{1,\dots,m\}$}
\]
and
\begin{equation}    \label{dJz1}
d\J(\bu)[{\bf v}] = \sum_{i}\frac{\partial\J}{\partial u_i}(\bu)[v_i]
 \qquad \hbox{for all} \ {\bf v} = (v_1,\dots,v_m) \in X.
\end{equation}
Moreover, direct computations imply that
\begin{equation}    \label{dJzstar}
\left\|\frac{\partial\J}{\partial u_i}(\bu)\right\|_{X_i'} \le \|d\J(\bu)\|_{X'}
\qquad \hbox{for all $i \in \{1,\dots,m\}$}
\end{equation}
and 
\begin{equation} \label{dJzstar2}
\|d\J(\bu)\|_{X'} \le \sum_{i}\left\|\frac{\partial\J}{\partial u_i}(\bu)\right\|_{X_i'}.
\end{equation}
Clearly, we have that
\[
d\J(\bu) = 0 \;\hbox{in $X$}\quad\iff\quad
\frac{\partial\J}{\partial u_i}(\bu)= 0 \;\hbox{in $X_i\quad$ for all } i\in\{1,\dots,m\}.
\]
\end{remark}

Now, we can state a regularity result.

\begin{proposition}\label{smooth1}
Suppose that conditions $(h_0)$--$(h_1)$, $(g_0)$ and \eqref{gucr} hold. 
Let $(\bu_n)_n \subset X$ and $\bu \in X$ be such that
\begin{equation}     \label{uconv}
\bu_n\to \bu\; \mbox{strongly in } W, \qquad \bu_n\to \bu\; \mbox{ a.e. in } \Omega
\quad \mbox{ if } n\to +\infty.
\end{equation}
If $M > 0$ exists such that
\begin{equation}       \label{ununiflim}
\| \bu_n\|_L\leq M \quad \hbox{for all $n\in\N$,}
\end{equation}
then
\[
\J(\bu_n) \to \J(\bu) \quad \mbox{ and } \quad 
\Vert d\J({\bf{u}_n}) - d\J(\bu)\Vert_{X^{\prime}}\to 0 \ \mbox{ as } n\to +\infty.
\]
Hence, $\J$ is a $C^1$ functional on $X$ with Fr\'echet differential  
defined as in (\ref{diff}).
\end{proposition}

\begin{proof}
Let $(\bu_n)_n\subset X$ and $\bu\in X$ be such that \eqref{uconv} and \eqref{ununiflim} hold. 
Taking any $i\in\{1,\dots, m\}$, from hypotheses $(h_0)$--$(h_1)$, \eqref{uconv} and \eqref{ununiflim},
by reasoning as in the proof of \cite[Proposition 3.1]{CP2}, it follows that the
``partial'' functional
\[
\mathcal{A}_i: u\in X_i\mapsto\int_{\Omega} A_i(x, u, \nabla u) dx \in \R
\]
is such that
\[
\mathcal{A}_i(u^n_{i})\rightarrow \mathcal{A}_i(u_i)\quad\mbox{ as } n\to +\infty.
\]
Moreover, by means of Dominated Convergence Theorem, conditions $(g_0)$, \eqref{Gsigma}, \eqref{uconv} and 
\eqref{ununiflim} imply that
\[
\int_{\Omega}G(x, \bu_n) dx \longrightarrow \int_{\Omega}G(x, \bu) dx.
\]
Thus, summing up, we have that
\[
\J(\bu_n)\rightarrow\J(\bu) \quad\mbox{ as } n\to +\infty.
\]
Now, we observe that \eqref{dJz1} gives
\[
\Vert d\J(\bu_n) -d\J(\bu)\Vert_{X^{\prime}}\leq\sum_{i} \left\Vert\frac{\partial\J}{\partial u_i}(\bu_n) 
-\frac{\partial\J}{\partial u_i}(\bu)\right\Vert_{X_i^{\prime}},
\]
so, in order to prove that $\Vert d\J({\bf{u}_n}) - d\J(\bu)\Vert_{X^{\prime}}\to 0$, it is enough to verify that
\begin{equation}   \label{forany}
\left\Vert\frac{\partial\J}{\partial u_i}(\bu_n) -\frac{\partial\J}{\partial u_i}(\bu)\right\Vert_{X_i^{\prime}}\rightarrow 0 
\quad\mbox{ for all } i\in\{1,\dots, m\}.
\end{equation}
To this aim, fixing any $i\in\{1,\dots,m\}$ and taking $v\in X_i$ such that $\Vert v\Vert_{X_i}\leq 1$, 
from \eqref{dJw} we have that
\[
\begin{split}
&\left\vert\frac{\partial\J}{\partial u_i}(\bu_n)[v] - \frac{\partial\J}{\partial u_i}(\bu)[v]\right\vert
\leq\int_{\Omega} \vert a_i(x, u^n_i, \nabla u^n_i) - a_i(x, u_i, \nabla u_i)\vert \vert\nabla v\vert dx\\
&\quad +\int_{\Omega}\vert A_{i, t}(x, u^n_i, \nabla u^n_i) - A_{i, t}(x, u_i, \nabla u_i)\vert dx +
\int_{\Omega}\vert G_{i}(x, \bu_n) - G_{i}(x, \bu)\vert dx,
\end{split}
\]
where, by reasoning as in \cite[Proposition 3.1]{CP2}, it can be proved that
\[
\int_{\Omega}\vert a_i(x, u^n_i, \nabla u^n_i) - a_i(x, u_i, \nabla u_i)\vert \vert\nabla v\vert dx\rightarrow 0
\quad\mbox{ uniformly with respect to } v 
\]
and 
\[
\int_{\Omega}\vert A_{i,t}(x, u^n_i, \nabla u^n_i) - A_{i,t}(x, u_i, \nabla u_i)\vert dx\to 0.
\]
Moreover, \eqref{uconv}, \eqref{ununiflim}, hypotheses $(g_0)$, \eqref{gucr}  
and, again, Dominated Convergence Theorem, imply that
\[
\int_{\Omega}\vert G_{i}(x, \bu_n) - G_{i}(x, \bu)\vert dx \rightarrow 0.
\]
Hence, summing up, we conclude that
\[
\left\vert\frac{\partial\J}{\partial u_i}(\bu_n)[v] -\frac{\partial\J}{\partial u_i}(\bu)[v]\right\vert\rightarrow 0 
\quad\mbox{ uniformly with respect to } v \in X_i, \; \|v\|_{X_i} \le 1,
\]
and, by the arbitrariness of $i\in\{1,\dots,m\}$, it follows that \eqref{forany} is satisfied, too.
\end{proof}


\section{The weak Cerami--Palais--Smale condition}
\label{sec_wcps}

In order to prove some more properties of functional $\J:X \to \R$ defined as
in \eqref{functional}, 
we require that not only $(h_0)$--$(h_1)$ hold but also
$R \geq 1$ exists such that for each $i\in\{1, \dots, m\}$
function $A_i(x,t,\xi)$ and its partial derivatives in \eqref{Ata} satisfy the following conditions:
\begin{itemize}
\item[$(h_2)$] a constant $\lambda >0$ exists such that
\[
a_i(x, t, \xi)\cdot\xi\geq\lambda \vert\xi\vert^{p_i} \quad \mbox{a.e. in } \Omega, \mbox{ for all } (t, \xi)\in\R\times\R^N
\]
with $p_i > 1$ as in $(h_1)$;
\item[$(h_3)$] some constants $\eta_1$, $\eta_2 > 0$ exist  such that
\begin{equation}\label{exh_2}
A_i(x, t, \xi)\leq \eta_1 a_i(x, t, \xi)\cdot\xi \quad \mbox{ a.e. in } \Omega \mbox{ if } \vert(t, \xi)\vert\geq R,
\end{equation}
\begin{equation}\label{exh_3}
\sup_{\vert (t, \xi)\vert\leq R} \vert A_i(x, t, \xi)\vert\leq\eta_2 \quad \mbox{ a.e. in } \Omega;
\end{equation}
\item[$(h_4)$] a constant $\mu_1>0$ exists such that
\[
a_i(x, t, \xi)\cdot\xi + A_{i, t}(x, t, \xi) t \geq \mu_1 a_i(x, t, \xi)\cdot\xi 
\quad \mbox{ a.e. in } \Omega \mbox{ if } \vert(t, \xi)\vert\geq R;
\]
\item[$(h_5)$] taking $p_i > 1$ as in $(h_1)$, some positive constants $\theta_i, \mu_2>0$ exist such that
\[
A_i(x, t, \xi) -\theta_i a_i(x, t, \xi)\cdot\xi -\theta_i A_{i, t}(x, t, \xi) t\geq\mu_2 a_i(x, t, \xi)\cdot\xi 
\]
a.e. in $\Omega$ if $\vert (t,\xi)\vert\geq R$, with 
\begin{equation}   \label{thi<pi}
\theta_i < \frac{1}{p_i};
\end{equation}
\item[$(h_6)$] for all $\xi, \xi^{\prime}\in\R^N$, with $\xi\neq\xi^{\prime}$, it is 
\[
[a_i(x, t, \xi) -a_i(x, t, \xi^{\prime})]\cdot[\xi-\xi^{\prime}] >0 \quad \mbox{ a.e. in } \Omega, \ \mbox{ for all } t\in\R.
\]
\end{itemize}

Moreover, let us assume that function $G(x,\bu)$ satisfies not only hypotheses $(g_0)$--$(g_1)$
but also the following Ambrosetti--Rabinowitz type condition:
\begin{itemize}
\item[$(g_2)$] taking $\theta_i$ as in $(h_5)$ for all $i\in \{1,\dots,m\}$, we have that
\[
0 <G(x, \bu)\leq\sum_{i} \theta_i G_{i}(x, \bu) u_i \quad 
\mbox{ if } \vert \bu\vert\geq R, \ \bu = (u_1,\dots,u_m),
\]
for a.e. $x\in\Omega$, with $R > 0$ as in the previous set of hypotheses $(h_3)$--$(h_5)$.
\end{itemize}

\begin{remark}    \label{remmu1}
We note that hypothesis $(h_4)$ is satisfied also if $t = 0$ and $\vert\xi\vert\geq R$,
then from $(h_2)$ we have $\mu_1\leq 1$. Furthermore, hypotheses $(h_4)$ and $(h_5)$ give
\begin{equation}    \label{A13}
A_i(x, t, \xi)\geq\left(\theta_i\mu_1 + \mu_2\right) a_i(x, t, \xi)\cdot \xi  
\quad \mbox{ a.e. in } \Omega \ \mbox{ if } \vert(t, \xi)\vert\geq R,
\end{equation}
whence, from $(h_2)$ we have that
\begin{equation}    \label{gezero}
A_i(x, t, \xi)\geq (\theta_i \mu_1 +\mu_2)\lambda \vert\xi\vert^{p_i}\geq 0 \quad 
\mbox{ a.e. in } \Omega \ \mbox{ if } \vert(t, \xi)\vert\geq R.
\end{equation}
Summing up, from \eqref{gezero} and assumption \eqref{exh_3} 
it follows that a positive constant $\eta_3$ exists such that
\begin{equation}    \label{4.4}
A_i(x, t, \xi)\geq (\theta_i \mu_1 +\mu_2)\lambda \vert\xi\vert^{p_i}-\eta_3 
\quad \mbox{ a.e. in } \Omega \mbox{ for all } (t, \xi)\in\R\times\R^N.
\end{equation}
\end{remark}

\begin{remark}
Taking $i\in\{1,\dots, m\}$, we note that \eqref{conto} in $(h_1)$ is not required as hypothesis
if $(h_2)$--$(h_5)$ and \eqref{conto2} hold.
In fact, in these assumptions not only \eqref{4.4} is satisfied but also direct computations 
imply that
\[
 A_i(x, t, \xi) \le \eta_1 \Phi_2^i(t) + \eta_2 
 + \eta_1 (\Phi_2^i(t) + \phi_2^i(t)) \vert\xi\vert^{p_i}
\]
a.e. in $\Omega$, for all $(t,\xi)\in\R\times\R^N$. Hence, $(h_1)$ can be replaced by the weaker condition
\begin{itemize}
\item[$(h_1')$] assumptions \eqref{conto1} and \eqref{conto2} hold.
\end{itemize}
Furthermore, taking $t = 0$ and $\vert\xi\vert\geq R$ in both \eqref{exh_2} and $(h_5)$, 
from $(h_2)$ it follows $\eta_1 \ge\mu_2 + \theta_i$ so, 
without loss of generality, we can assume that $\mu_2$ is so that 
\[
\eta_1 - \mu_2 - \theta_i > 0.
\]
Thus, in order to give better growth conditions on function $A_i(x, t, \xi)$,
hypotheses \eqref{exh_2} and $(h_5)$ give
\begin{equation}   \label{Alast}
\left(\frac{\eta_1-\mu_2-\theta_i}{\eta_1 \theta_i}\right) A_i(x, t, \xi)
\geq A_{i, t}(x, t, \xi) t \quad \mbox{ a.e. in } \Omega \ \mbox{ if } \vert(t, \xi)\vert\geq R.
\end{equation}
Hence, \eqref{gezero}, \eqref{Alast} with hypotheses \eqref{conto2}, \eqref{exh_2} 
and direct computations imply that $\eta_4 >0$ exists such that
\[
A_i(x, t, \xi) \leq \eta_4 \vert t\vert^{\frac{\eta_1-\mu_2-\theta_i}{\eta_1\theta_i}}\vert\xi\vert^{p_i}
\quad \mbox{ a.e. in } \Omega \ \mbox{ if } \vert t\vert\geq 1, \; \vert\xi\vert\geq R,
\]
then \eqref{A13} gives
\[
a_i(x, t, \xi)\cdot\xi\leq \frac{\eta_4}{\theta_i \mu_1 +\mu_2}
\vert t\vert^{\frac{\eta_1-\mu_2-\theta_i}{\eta_1\theta_i}}\vert\xi\vert^{p_i}
\quad \mbox{ a.e. in } \Omega \ \mbox{ if } \vert t\vert\geq 1, \; \vert\xi\vert\geq R.
\]
\end{remark}

\begin{remark}   \label{RemarkG}
Taking $i\in\{1,\dots, m\}$, from $(g_2)$ we have that
\[
0< G(x, u \,{\bf e}_i)\leq \theta_i G_{i}(x, u\, {\bf e}_i) u \quad\mbox{ for a.e. } x\in\Omega, 
\; \mbox{ if }\ u \in \R,\, \vert u\vert\geq R.
\]
Hence, $(g_0)$ and direct computations imply that $h_i\in L^{\infty}(\Omega)$, 
$h_i(x) >0$ for a.a. $x\in\Omega$, exists such that
\begin{equation}    \label{Gthetai}
G(x, u\, {\bf e}_i)\geq h_i(x) \vert u\vert^{\frac{1}{\theta_i}} \quad\mbox{ for a.a. } x\in\Omega,
\; \mbox{ if }  \ u \in \R,\ \vert u\vert\geq R.
\end{equation}
Thus, if also \eqref{gucr} holds, from \eqref{Gsigma}, \eqref{thi<pi} 
and \eqref{Gthetai} not only we obtain that 
\begin{equation}    \label{minimax}
1 < p_i <\frac{1}{\theta_i}\leq q_i
\end{equation}
but also
\begin{equation}    \label{Gthetaibis}
G(x, u\, {\bf e}_i)\geq h_i(x) \vert u\vert^{\frac{1}{\theta_i}} - \sigma_i
\quad\mbox{ for a.a. } x\in\Omega,
\; \mbox{ all } u\in \R
\end{equation}
for a suitable $\sigma_i >0$.
\end{remark}

\begin{example}   \label{ex0}
Let us consider
\[
G(x, \bu) = \sum_{i} c_i\vert u_i\vert^{q_i} +
c_{\ast} \prod_{i =1}^m\vert u_i\vert^{\gamma_i},
\]
where we assume that
\[
q_i >1, \quad \gamma_i >1 \qquad\hbox{for all $i\in\{1,\dots, m\}$.}
\]
Then, $(g_0)$ is verified and for any $i\in\{1,\dots, m\}$ we have
\[
G_{i}(x, \bu) = c_i q_i\vert u_i\vert^{q_i-2} u_i +
c_{\ast}\gamma_i\vert u_i\vert^{\gamma_i-2} u_i\prod_{\substack{j =1\\j\neq i}}^m \vert u_j\vert^{\gamma_j} 
\quad\mbox{ for a.e. } x\in\Omega.
\]
If, in addition, we suppose that 
\[
\gamma_i < q_i \qquad\hbox{for all $i\in\{1,\dots, m\}$,}
\]
then the generalized Young inequality and direct computations allow us to conclude that
also \eqref{gucr} holds by taking
\[
s_{i,j} = (m-1) \gamma_j \frac{q_i-1}{q_i-\gamma_i} \qquad  \hbox{for all $i$, $j \in\{1,\dots, m\}$, $j\neq i$.}
\]
Finally, if we have that
\[
\sum_{i} \frac{\gamma_i}{q_i}\geq 1, 
\]
then hypothesis $(g_2)$ is verified, too, by taking each $\theta_i \ge \frac{1}{q_i}$.
\end{example}

Up to now, no upper bound is required for the growth of the nonlinear term $G(x, \bu)$. Anyway,
the subcritical assumptions \eqref{crit_exp} and \eqref{crit_expi}
are required for proving the weak Cerami--Palais--Smale condition.

\begin{remark}   \label{rmkcrit}
Let $i$, $j\in\{1,\dots,m\}$, $j \ne i$, be such that
$s_{i, j}$ verifies condition \eqref{crit_expi}. Then, $\tilde{q}_i>0$ exists such that
\begin{equation}  \label{crits11}
1 < \tilde{q}_i <p_i^{\ast}, \quad 0 \le \ws_{i, j}:=\frac{s_{i, j} \tilde{q}_i}{\tilde{q}_i-1} < p_j^{\ast}.
\end{equation}
In fact, if both $p_i<N$ and $p_j<N$, then
\eqref{crit_expi} implies that
\[
p_i p_j^{\ast} - N s_{i, j} >0 \quad\mbox{ and }\quad 
1 \le \frac{p_j^{\ast}}{p_j^{\ast} - s_{i, j}} 
\le \frac{p_i p_j^{\ast}}{p_ip_j^{\ast} - N s_{i, j}} <p_i^{\ast},
\]
hence, $\tilde{q}_i$ exists such that
\begin{equation}  \label{crits12}
\frac{p_i p_j^{\ast}}{p_i p_j^{\ast} - N s_{i, j}} < \tilde{q}_i < p_i^{\ast}
\end{equation}
so that both estimates in \eqref{crits11} hold. On the other hand, 
if $p_i \ge N$, respectively $p_j \ge N$, with $p_i^{\ast} = +\infty$, 
respectively $p_j^{\ast} = +\infty$, and \eqref{0infinito}
imply that the conditions in \eqref{crits11} are less restrictive and 
the existence of $\tilde{q}_i$ and $\ws_{i,j}$ is easier to prove.  \\
\end{remark}

\begin{remark}
Assume that $(g_0)$ and $(g_1)$ hold.
If for every $i$, $j\in\{1,\dots, m\}$, $j\ne i$, by using the same notations in 
Remark \ref{rmkcrit}, from \eqref{crits11} and Young inequality it follows that
\begin{equation}    \label{crit_Young}
\vert u_i\vert\ \vert u_j\vert^{s_{i, j}}\ \leq\ 
\frac{\vert u_i\vert^{\tilde{q}_i}}{\tilde{q}_i} +\frac{\tilde{q}_i-1}{\tilde{q}_i}\vert u_j\vert^{\ws_{i, j}}
\le \ \vert u_i\vert^{\tilde{q}_i} + \vert u_j\vert^{\ws_{i, j}}
\quad \hbox{for all $\bu \in \R^m$.}
\end{equation}
Thus, from \eqref{Gsigma} and \eqref{crit_Young} we obtain that
\[
\vert G(x, \bu)\vert\leq \sigma_1 \sum_{i} \Big(1 + \vert u_i\vert^{q_i} 
+\sum_{j\neq i} (\vert u_i\vert^{\tilde{q}_i} +\vert u_j\vert^{\ws_{i, j}})\Big) 
\quad \hbox{for all $\bu \in \R^m$.}
\]
Hence, setting
\begin{equation}    \label{qsegnatoi}
\overline{q}_i = \max\{ q_i, \tilde{q}_i, \ws_{j,i}:\ 1 \le j\le m,\ j\ne  i\} 
\quad \mbox{ for any } i\in\{1,\dots, m\},
\end{equation}
it follows that
\begin{equation}     \label{Gsigmamax}
\vert G(x, \bu)\vert\leq \sigma_0 \sum_{i} (1+ \vert u_i\vert^{\overline{q}_i})
\quad \hbox{for all $\bu \in \R^m$}
\end{equation}
for a suitable constant $\sigma_0 \ge \sigma_1$.\\
At last, if $(g_2)$ is verified too, then from \eqref{minimax}
and \eqref{crits11} we obtain that
\begin{equation} \label{non_so}
1 < p_i <\frac{1}{\theta_i} \le \overline{q}_i < p_i^{\ast}
\qquad \hbox{for all $i \in \{1,\dots,m\}$.}
\end{equation}
\end{remark}

Now, we show that the $(wCPS)$-- condition holds. To this aim,
the following boundedness result is required (for its proof, see \cite[Theorem II.5.1]{LU}).

\begin{lemma}    \label{Ladyz}
Let $\Omega$ be an open bounded subset of $\R^N$ and consider $u\in W_0^{1, p}(\Omega)$ with $p \le N$. 
Suppose that $\gamma > 0$ and $k_0\in\N$ exist such that 
\[
\int_{\Omega_{k}^{+}} \vert\nabla u\vert^p dx\ \leq\ 
\gamma\left( \int_{\Omega^{+}_{k}} (u - k)^r dx\right)^{\frac{p}{r}} 
+ \gamma \sum_{l=1}^\nu k^{\alpha_l} [\meas(\Omega^{+}_{k})]^{1-\frac{p}{N}+\varepsilon_l}
\]
for all $k\ge k_0$, with $\Omega^{+}_{k} :=\lbrace x\in\Omega : u(x) > k\rbrace$ and $r, \nu, \alpha_l, \varepsilon_l$ 
are positive constants such that
\[
1\leq r < p^{\ast}, \qquad \varepsilon_l > 0, \qquad p\leq\alpha_l < \varepsilon_l p^{\ast} + p.
\]
Then, $\displaystyle\esssup_{\Omega} u$ is bounded from above by a positive constant 
which can be chosen so that it depends only on 
$\meas(\Omega), N, p, \gamma, k_0, r, \nu, \varepsilon_l, \alpha_l, \vert u\vert_{p^{\ast}}$ (eventually,
$\vert u\vert_{q}$ for some $q > r$ if $p^{\ast} = +\infty$).
\end{lemma}

\begin{proposition}     \label{PropwCPS}
Assume that hypotheses $(h_0)$--$(h_6)$ and $(g_0)$--$(g_2)$ hold. 
Then functional $\J$ in \eqref{functional} satisfies condition $(wCPS)$ in $\R$.
\end{proposition}

\begin{proof}
Taking $\beta\in\R$, let $(\bu_n)_n \subset X$, with $\bu_n = (u^n_1,\dots, u^n_m)$, be a sequence such that
\begin{equation}  \label{3.15}
\J(\bu_n)\to \beta \quad \mbox{ and }\quad 
\left\Vert d\J(\bu_n)\right\Vert_{X^{\prime}}\left(1 + \Vert \bu_n\Vert_{X}\right)\to 0.
\end{equation}
We have to prove that $\bu=(u_1,\dots, u_m)\in X$ exists such that 
\begin{enumerate}
\item[$(i)$] $\bu_n \to \bu$ strongly in $W$;
\item[$(ii)$] $\J(\bu) = \beta, \  d\J(\bu) = 0$.
\end{enumerate}
To this aim, our proof is divided in the following steps:\\
$1.\;$ $(\bu_n)_n$ is bounded in $W$; hence, up to subsequences, 
$\bu\in W$ exists such that 
\begin{equation}         \label{3.16}
\bu_n\rightharpoonup \bu\ \mbox{weakly in $W$,}
\end{equation}
i.e., $u_i^n \rightharpoonup u_i$ weakly in $W_i$ for all $i \in \{1,\dots,m\}$,   
\begin{equation}  \label{3.17}
u_i^n \to u_i\ \mbox{strongly in } L^{r_i}(\Omega)\; \mbox{if $r_i\in [1, p_i^{\ast}[$,}
\quad\hbox{for all $i \in \{1,\dots,m\}$,}\\
\end{equation}
\begin{equation}        \label{3.18}
\bu_n\to \bu \mbox{ a.e. in } \Omega;
\end{equation}
$2.\;$ $\bu\in L$, then $\bu\in X$;\\
$3.\;$for any $k>0$, let $T_k:\R\rightarrow\R$ and $\T_k: \R^m \to\R^m$
be defined as
\[   
T_k t:= \begin{cases}
t &\hbox{ if } \vert t\vert\leq k\\
k \frac{t}{\vert t\vert} &\hbox{ if } \vert t\vert > k
\end{cases}
\]
and 
\begin{equation}  \label{Tkappa}
\T_k(t_1,\dots, t_m) := (T_k t_1,\dots, T_k t_m),
\end{equation}
then, taking any $k\geq \max\lbrace\| \bu\|_{L}, R\rbrace + 1$ (with $R\geq 1$ as in our set of hypotheses), 
we have
\begin{equation}      \label{step3.1}
\Vert d\J (\T_k\bu_n)\Vert_{X^{\prime}}\to 0
\end{equation}
and
\begin{equation}      \label{step3.2}
\J (\T_k\bu_n)\to \beta;
\end{equation}
$4.\;$ $\Vert\T_k \bu_n - \bu\Vert_{W}\to 0$, then $(i)$ holds;\\
$5.\;$ $(ii)$ is satisfied.
\smallskip

\noindent
For simplicity, here and in the following we will use the notation $(\varepsilon_n)_n$ 
for any infinitesimal sequence depending only on sequence $(\bu_n)_n$, 
$(\varepsilon_{k,n})_n$ for any infinitesimal sequence depending not only on $(\bu_n)_n$ 
but also on a fixed integer $k$. Moreover, $b_l$ will denote any strictly positive constant independent of $n$.\\
{\sl Step 1.} Firstly, we observe that \eqref{dJzstar} and \eqref{3.15} imply that
\begin{equation}\label{parti}
\frac{\partial\J}{\partial u_i}(\bu_n)[u^n_i] = \varepsilon_n
\quad \hbox{for all $i\in\{1,\dots, m\}$.}
\end{equation}
Thus, taking $\theta_i$, $i\in\{1,\dots, m\}$, as in hypotheses $(h_5)$ and $(g_2)$,
and fixing $n\in\N$, from \eqref{functional}, \eqref{dJw}, \eqref{3.15} and \eqref{parti} we have that
\[
\begin{split}
&\beta +\varepsilon_n =\J(\bu_n) -\sum_{i} \theta_i\frac{\partial\J}{\partial u_i}(\bu_n)[u^n_i]\\
&\quad=\sum_{i} \int_{\Omega} \big(A_i(x, u^n_i, \nabla u^n_i) -\theta_i a_i(x, u^n_i, \nabla u^n_i)\cdot\nabla u^n_i
-\theta_i A_{i, t}(x, u^n_i,\nabla u^n_i) u^n_i\big) dx\\
&\qquad +\int_{\Omega} \Big(\sum_{i} \big(\theta_i G_{i}(x, \bu_n) u^n_i\big) - G(x, \bu_n)\Big) dx.
\end{split}
\]
Now, fix $i\in\{1,\dots, m\}$ and set
\[
\Omega_{n,R}^i =\lbrace x\in\Omega:\, \vert(u^n_i(x), \nabla u^n_i(x))\vert >R\rbrace.
\]
Then, we have that
\begin{equation}    \label{misura}
\int_{\Omega\setminus\Omega_{n,R}^i}\vert\nabla u^n_i\vert^{p_i} dx\leq R^{p_i} \meas(\Omega).
\end{equation}
On the other hand, hypothesis $(h_1)$ implies that
\[
\begin{split}
\sum_{i}\int_{\Omega\setminus\Omega_{n, R}^i}| A_i(x, u^n_i, \nabla u^n_i)&-\theta_i a_i(x, u^n_i, \nabla u^n_i)\cdot\nabla u^n_i\\
&-\theta_i A_{i, t}(x, u^n_i, \nabla u^n_i) u^n_i| dx\leq b_1,
\end{split}
\]
while from assumptions $(h_2)$ and $(h_5)$ it follows that
\[
\begin{split}
&\sum_{i} \int_{\Omega_{n,R}^i} \big(A_i(x, u^n_i, \nabla u^n_i) u^n_i 
-\theta_i a_i(x, u^n_i, \nabla u^n_i)\cdot\nabla u^n_i -\theta_i A_{i, t}(x, u^n_i, \nabla u^n_i) u^n_i\big) dx\\
&\qquad\geq\mu_2\sum_{i} \int_{\Omega_{n,R}^i} a_i(x, u^n_i, \nabla u^n_i)\cdot\nabla u^n_i dx
\geq \mu_2\lambda\sum_{i}\int_{\Omega_{n, R}^i}\vert\nabla u^n_i\vert^{p_i} dx.
\end{split}
\]
Moreover, from hypotheses \eqref{gucr}, $(g_2)$ with \eqref{Gsigma} and direct computations we have that
\[
\int_{\Omega}\Big(\Big(\sum_{i} \theta_i G_{i}(x, \bu_n) u^n_i\Big) - G(x, \bu_n) \Big) dx \geq - b_2.
\]
Thus, summing up, from the previous estimates and \eqref{misura}, we obtain that
\[
\beta +\varepsilon_n \geq
\mu_2\lambda\sum_{i} \int_{\Omega_{n, R}^i}\vert\nabla u^n_i\vert^{p_i} dx - b_3
\geq \mu_2 \lambda\sum_{i} \Vert u^n_i\Vert_{W_i}^{p_i} - b_4.
\]
Hence, $(\bu_n)_n$ is bounded in $W$ and so $\bu\in W$ exists such that, up to subsequences,
\eqref{3.16}--\eqref{3.18} hold.\\
{\sl Step 2.} In order to prove that $\bu\in L$, arguing by contradiction, we assume that 
$i\in\{1,\dots, m\}$ exists such that $p_i < N$ (the proof if $p_i =N$ is simpler)
 and $u_i\notin L^{\infty}(\Omega)$, then either
\begin{equation}    \label{sup_u}
\esssup_{\Omega} u_i =+\infty
\end{equation}
or
\begin{equation}  \label{sup_menou}
\esssup_{\Omega} (-u_i) =+\infty.
\end{equation}
If \eqref{sup_u} holds, then for any fixed $k\in\N$ we have that
\begin{equation}   \label{meas>0}
\meas (\Omega^{i,+}_k) >0,\quad 
\hbox{with $\Omega^{i,+}_k =\{x\in\Omega:\ u_i(x) >k\}$.}
\end{equation}
Defining $R^{+}_k: \R\rightarrow \R$ as
\begin{equation}\label{erre_defn}
R^{+}_k t : = \begin{cases}
0 &\hbox{ if } t\leq k\\
t - k &\hbox{ if } t > k
\end{cases},
\end{equation}
from \eqref{3.16} it follows that
\[
R^{+}_k u^n_i\rightharpoonup R^{+}_k u_i\quad\mbox{weakly in } W_i,
\]
whence the sequentially weakly lower semicontinuity of $\Vert\cdot\Vert_{W_i}$ gives
\begin{equation}   \label{liminf1}
\int_{\Omega^{i,+}_k} \vert\nabla u_i\vert^{p_i} dx\leq 
\liminf_{n}\int_{\Omega^{i,+}_{n, k}}\vert\nabla u^n_i\vert^{p_i} dx,
\end{equation}
with $\Omega^{i,+}_{n, k} =\{x\in\Omega:\ u^n_i(x) > k\}$.
On the other hand, from \eqref{dJzstar}, \eqref{3.15} and definition \eqref{erre_defn} 
we have that
\[
\frac{\partial\J}{\partial u_i}(\bu_n)[R^{+}_k u^n_i]\rightarrow 0,
\]
so, from \eqref{meas>0} an integer $n_k\in\N$ exists such that
\begin{equation}    \label{>nk}
\frac{\partial\J}{\partial u_i}(\bu_n)[R^{+}_k u^n_i] < \meas(\Omega^{i,+}_k) 
\quad\mbox{ for all } n\geq n_k.
\end{equation}
Now, taking $k >R$ ($R$ as in our set of hypotheses) and $n \in \N$,
as  $\mu_1\leq 1$ (see Remark \ref{remmu1}) from \eqref{dJw}, $(h_2)$, $(h_4)$ 
and direct calculations it follows that
\[
\begin{split}
\frac{\partial\J}{\partial u_i}(\bu_n)[R^{+}_k u^n_i] =
&\int_{\Omega^{i,+}_{n,k}} \Big(1 - \frac{k}{u^n_i}\Big)
(a_i(x, u^n_i,\nabla u^n_i)\cdot\nabla u^n_i + A_{i, t}(x, u^n_i, \nabla u^n_i) u^n_i) dx \\
&\ + \int_{\Omega^{i,+}_{n,k}} \frac{k}{u^n_i}\ a_i(x, u^n_i,\nabla u^n_i)\cdot\nabla u^n_i dx 
- \int_{\Omega} G_{i}(x, \bu_n) R^{+}_k u^n_i dx\\
\geq\ &\mu_1 \int_{\Omega^{i,+}_{n, k}}a_i(x, u^n_i,\nabla u^n_i)\cdot\nabla u^n_i dx 
-\int_{\Omega} G_{i}(x, \bu_n) R^{+}_{k} u^n_i dx\\
\ge\ &\mu_1\lambda\int_{\Omega^{i,+}_{n,k}}\vert\nabla u^n_i\vert^{p_i} dx -
\int_{\Omega} G_{i}(x, \bu_n) R^{+}_{k} u^n_i dx,
\end{split}
\]
which, together with \eqref{>nk}, implies that
\begin{equation}    \label{mu1lam}
\mu_1\lambda\int_{\Omega^{i,+}_{n, k}} \vert\nabla u^n_i\vert^{p_i} dx
\leq\meas(\Omega^{i,+}_k) +\int_{\Omega} G_{i}(x, \bu_n) R^{+}_k u^n_idx
\end{equation}
for all $n\ge n_k$.
We note that
\begin{equation}   \label{DCTG}
\int_{\Omega} G_{i}(x, \bu_n) R^{+}_k u^n_i dx\longrightarrow 
\int_{\Omega} G_{i}(x, \bu) R^{+}_k u_i dx.
\end{equation}
In fact, from $(g_0)$ and \eqref{3.18} we obtain that
\[
G_{i}(x, \bu_n) R^{+}_k u^n_i\rightarrow G_{i}(x, \bu) R^{+}_k u_i \quad\mbox{ a.e. in } \Omega,
\]
while, since \eqref{crit_expi} implies that suitable exponents can be choosen 
so that \eqref{crits11} holds, from \eqref{gucr}, \eqref{crit_Young} and also
\eqref{crit_exp}, \eqref{crits11}, \eqref{3.17}, \cite[Theorem 4.9]{Br} 
it follows that a function $h\in L^{1}(\Omega)$ exists such that
\[
\vert G_{i}(x, \bu_n) R^{+}_k u^n_i\vert 
\le \sigma \Big(\vert u^n_i\vert +\vert u^n_i\vert^{q_i} + 
\sum_{j\neq i} (\vert u_i^{n}\vert^{\tilde{q}_{i}} +\vert u_j^n\vert^{\ws_{i,j}})\Big) \leq h(x) 
\quad \mbox{ a.e. in } \Omega,
\]
then Dominated Convergence Theorem applies. \\
Thus, summing up, from \eqref{liminf1}, \eqref{mu1lam}, \eqref{DCTG} and, again, 
\eqref{gucr} and \eqref{crit_Young} we have that \eqref{crit_expi} and direct computations
imply that
\begin{equation}    \label{quasicompl}
\int_{\Omega^{i,+}_k}\vert \nabla u_i\vert^{p_i} dx\leq 
b_5 \Big(\meas(\Omega^{i,+}_k) +\int_{\Omega^{i,+}_k}\vert u_i\vert^{\overline{q}_i} dx + 
\sum_{j\neq i} \int_{\Omega^{i,+}_k} \vert u_j\vert^{\ws_{i,j}}dx\Big),
\end{equation}
with $\overline{q}_i$ as in \eqref{qsegnatoi} and $\ws_{i,j}$ as in \eqref{crits11}.\\
At last, H\"older inequality and \eqref{crits11} imply that
\[
\int_{\Omega^{i,+}_k} \vert u_j\vert^{\ws_{i,j}}dx \le 
|u_j|_{p^*_j}^{\ws_{i,j}} [\meas(\Omega^{i,+}_k)]^{1-\frac{\ws_{i,j}}{p^*_j}}
\quad \hbox{for each $j \ne i$,}
\]
while \eqref{non_so} and direct computations give
\[
\int_{\Omega^{i,+}_k}\vert u_i\vert^{\overline{q}_i} dx \le 
2^{\overline{q}_i - 1} |u_i|_{\overline{q}_i}^{\overline{q}_i-p_i} 
\Big(\int_{\Omega^{i,+}_k} (u_i-k)^{\overline{q}_i} dx\Big)^{\frac{p_i}{\overline{q}_i}}
+ 2^{\overline{q}_i - 1} k^{\overline{q}_i} \meas(\Omega^{i,+}_k),
\]
so from \eqref{quasicompl}, Sobolev Embedding Theorems and direct computations
we obtain that
\begin{equation}    \label{compl}
\begin{split}
\int_{\Omega^{i,+}_k}\vert \nabla u_i\vert^{p_i} dx \leq & \ b_6 \Big(\int_{\Omega^{i,+}_k} (u_i-k)^{\overline{q}_i} dx\Big)^{\frac{p_i}{\overline{q}_i}}\\
& + b_6\sum_{j} k^{\alpha_j}  [\meas(\Omega^{i,+}_k)]^{1-\frac{p_{i}}{N} + \eps_j},
\end{split}
\end{equation}
with $b_6 = b_6(\|\bu\|_W) > 0$, where we set
\[
\alpha_j = \left\{\begin{array}{ll}
\overline{q}_i &\hbox{if $j=i$}\\
0 &\hbox{if $j\ne i$}
\end{array}\right.,\qquad
\eps_j = \left\{\begin{array}{ll}
\frac{p_i}{N} &\hbox{if $j=i$}\\
\frac{p_i}{N} - \frac{\ws_{i,j}}{p^*_j} &\hbox{if $j\ne i$}
\end{array}\right. .
\]
From \eqref{crits12} it follows that $\eps_j > 0$ for each $j \ne i$, so  
\eqref{compl} with \eqref{non_so} (here, $\eps_i p^*_i + p_i = p^*_i$) 
allows us to apply Lemma \ref{Ladyz} and $u_i$
is essentially bounded from above in contradiction with \eqref{sup_u}.\\
Similar arguments make us to rule out also \eqref{sup_menou}; hence, it has to be $u_i \in L^\infty(\Omega)$.\\
{\sl Step 3.} In order to prove this statement, we extend the 
main arguments used in the corresponding {\sl Step 3} in
the proof of \cite[Proposition 4.6]{CSS} to our setting
but introducing some technical changes as in \cite[Proposition 4.6]{CP2}.
Anyway, for the sake of completeness, here we give the main tools.\\
Taking $k\geq \max\lbrace\| \bu\|_{L}, R\rbrace + 1$, we define
$R_k: \R \to \R$ and $\er_k:\R^m \to \R^m$ as
\[
R_k t= t - T_k t = \begin{cases}
0 &\hbox{ if } |t|\leq k\\
t - k\frac{t}{|t|} &\hbox{ if } |t|>k
\end{cases} ,
\]
\[
\er_k(t_1,\dots, t_m) = (R_k t_1,\dots, R_k t_m),
\]
and denote
\[
\Omega^{i}_{n, k}:=\lbrace x\in\Omega: \vert u^n_i(x)\vert > k\rbrace \qquad
\hbox{for any $n \in \N$, $i\in\{1,\dots, m\}$.}
\]
By definition, it follows that
\begin{equation}    \label{cappa}
\Vert\T_k\bu_n\Vert_X\leq\Vert\bu_n\Vert_X \quad\mbox{ and }\quad
\Vert\er_k \bu_n\Vert_X\leq\Vert\bu_n\Vert_X \quad \mbox{ for all $n\in\N$.}
\end{equation}
Moreover, we have that $\T_k\bu = \bu$ and $\er_k\bu = \textbf{0}$ a.e. in $\Omega$, so
 from \eqref{3.16}--\eqref{3.18}, in particular, it follows that
\begin{equation}   \label{Tkae}
\T_k\bu_n \to \bu \ \mbox{ a.e. in } \Omega,
\end{equation}
\begin{equation}
\label{Rkstrongly}
\er_k \bu_n\to \textbf{0} \quad
\mbox{ in } L^{r_1}(\Omega)\times\dots\times L^{r_m}(\Omega)
\end{equation}
for any  $(r_1,\dots,r_m) \in [1, p_1^{\ast}[ \times\dots\times [1, p_m^{\ast}[$,
\begin{equation}
\label{limmeas0}
\meas(\Omega^{i}_{n, k})\to 0\quad \mbox{ for all $i\in\{1,\dots, m\}$.}
\end{equation}
Fixing any $i\in\{1,\dots, m\}$, from \eqref{dJzstar}, \eqref{3.15} 
and \eqref{cappa} we have that
\begin{equation}   \label{normRki}
\left\Vert\frac{\partial\J}{\partial u_i}(\bu_n)\right\Vert_{X^{\prime}_i} \Vert R_k u^n_i\Vert_{X_i}\to 0.
\end{equation}
Then, reasoning as in the previous \emph{Step 2} 
but replacing  $R^{+}_k u_i$ with $R_k u_i$, from \eqref{normRki}
it follows that
\begin{equation}\label{come2u}
\begin{split}
\eps_n + \int_{\Omega}G_{i}(x, \bu_n) R_k u^n_i dx \ & = \frac{\partial\J}{\partial u_i}(\bu_n) [R_{k} u^n_i] + \int_{\Omega}G_{i}(x, \bu_n) R_k u^n_i dx\\
&\geq \mu_1\int_{\Omega^{i}_{n, k}} a_i(x, u^n_i, \nabla u^n_i)\cdot\nabla u^n_i dx\\
&\geq \lambda\mu_1\int_{\Omega^{i}_{n, k}}\vert\nabla u^n_i\vert^{p_i}dx.
\end{split}
\end{equation}
Since arguments similar to those ones used for proving \eqref{DCTG} apply, 
from \eqref{Rkstrongly} we obtain that 
\[
\int_{\Omega}G_{i}(x, \bu_n) R_k u^n_i dx \longrightarrow 0,
\]
then \eqref{come2u} implies both
\begin{equation}   \label{Rk_0}
\int_{\Omega^{i}_{n, k}} \vert\nabla u^n_i\vert^{p_i} dx\longrightarrow 0, 
\quad\mbox{ i.e. }\quad \Vert R_k u^n_i\Vert_{W_i} \to 0,
\end{equation}
and
\begin{equation}    \label{aizero}
\int_{\Omega^{i}_{n, k}} a_i(x, u^n_i, \nabla u^n_i)\cdot\nabla u^n_i dx \longrightarrow 0.
\end{equation}
By means of \eqref{dJzstar2}, if we prove that
\[
\left\Vert\frac{\partial\J}{\partial u_i} (\T_k \bu_n)\right\Vert_{X^{\prime}_i}\to 0 
\quad\mbox{ for all } i\in\{1,\dots, m\},
\]
then \eqref{step3.1} holds. 
To this aim, fixing any $i\in\{1,\dots, m\}$, we take $v \in X_i$ such that $\Vert v\Vert_{X_i} =1$. 
Direct computations imply that
\begin{equation}   \label{sumup}
\begin{split}
\frac{\partial\J}{\partial u_i}(\T_k\bu_n)[v] & \ = \frac{\partial\J}{\partial u_i}(\bu_n)[v]\\
&\quad-\int_{\Omega_{n, k}^{i}} \left(a_i(x, u^n_i, \nabla u^n_i)\cdot\nabla v + A_{i, t}(x, u^n_i, \nabla u^n_i) v \right) dx \\
&\quad+ \int_{\Omega^{i}_{n, k}}\left( a_i(x, T_k u^n_i, 0)\cdot\nabla v + A_{i, t}(x, T_k u^n_i, 0) v\right) dx\\
&\quad+ \int_{\Omega}(G_{i}(x, \bu_n) - G_{i}(x, \T_k\bu_n)) v \ dx.
\end{split}
\end{equation}
We observe that $(h_1)$ with $\xi = 0$ and $|T_ku_i^n| \le k$ for all $n \in \N$
imply the boundness of both $A_{i,t}(x, T_k u^n_i, 0)$ and $a_i(x, T_k u^n_i, 0)$ 
in set $\Omega^{i}_{n, k}$; hence, from \eqref{limmeas0} and direct computations 
so to ``erase'' $v$ from the limit, we obtain that
\begin{equation}   \label{unifvi}
\int_{\Omega^{i}_{n, k}}\left( a_i(x, T_k u^n_i, 0)\cdot\nabla v +
 A_{i, t}(x, T_k u^n_i, 0) v\right) dx\rightarrow 0
\end{equation}
uniformly with respect to $v$. On the other hand, Dominated Convergence Theorem implies that
\begin{equation}  \label{GDom}
\int_{\Omega}\vert G_{i}(x, \bu_n) - G_{i}(x, \T_k\bu_n)\vert dx\rightarrow 0.
\end{equation}
In fact, we have that
\[
\begin{split}
\int_{\Omega}|G_{i}(x, \bu_n) - G_{i}(x, \T_k\bu_n)| dx
 \leq\ &\int_{\Omega}\vert G_{i}(x, \bu_n) - G_{i}(x, \bu)\vert dx\\
& \ +\int_{\Omega}\vert G_{i}(x, \T_k\bu_n) - G_{i}(x, \bu)\vert dx,
\end{split}
\]
where $(g_0)$ together with \eqref{3.18}, respectively \eqref{Tkae}, give
\[
G_{i}(x, \bu_n)\rightarrow G_{i}(x, \bu) \quad\mbox{ and }\quad 
G_{i}(x, \T_k\bu_n)\rightarrow G_{i}(x, \bu) \quad \hbox{a.e. in $\Omega$,}
\]
and from \eqref{gucr}, Young inequality, \eqref{crit_exp}, \eqref{crits11},
\eqref{3.17} and \cite[Theorem 4.9]{Br} we obtain that
\[
|G_{i}(x, \bu_n)|\leq b_7\left( 1+\vert u^n_i\vert^{q_i} +
\sum_{j\neq i}\vert u_j^n\vert^{\ws_{i, j}}\right) \le \bar{h}(x) \quad\mbox{ a.e. in $\Omega$}
\]
for a suitable $\bar{h} \in L^1(\Omega)$, while, again, \eqref{gucr} and the boundedness of 
$(\T_k\bu_n)_n$ in $L$ give
\[
\vert G_{i}(x, \T_k\bu_n)\vert\leq b_8 \quad\mbox{ a.e. in $\Omega$,}
\]
so Dominated Convergence Theorem applies and \eqref{GDom} holds. \\
Thus, summing up, from \eqref{sumup}--\eqref{GDom} we obtain that
\begin{equation}   \label{last_int}
\begin{split}
&\left\vert\frac{\partial\J}{\partial u_i}(\T_k\bu_n)[v]\right\vert\leq\varepsilon_{k,n}\\
&\qquad\qquad\qquad\quad+\left\vert\int_{\Omega^{i}_{n,k}}(a_i(x, u^n_i, \nabla u^n_i)\cdot\nabla v +A_{i, t}(x, u^n_i, \nabla u^n_i) v)dx\right\vert,
\end{split}
\end{equation}
where $\varepsilon_{k, n}$ is independent of $v$. 
At last, the estimate of the last integral in \eqref{last_int} can be obtained 
as in the corresponding {\sl Step 3} in the proof of \cite[Proposition 4.6]{CP2} 
but testing $\frac{\partial\J}{\partial u_i}(\T_k\bu_n)$ on the new test functions 
$v R^{+}_k u^n_i$ and $v R^{-}_k u^n_i$, with 
$R^{-}_k: \R \to \R$ such that
\[
R_k^- t= \begin{cases}
0 &\hbox{ if } t \ge - k\\
t + k &\hbox{ if } t < -k
\end{cases}.
\]
Then, \eqref{step3.1} holds.\\
Finally, we have to prove \eqref{step3.2}. 
From \eqref{functional}, \eqref{Tkappa} and direct computations, we have that
\[
\begin{split}
\J(\T_k\bu_n) &= \J(\bu_n) -
\sum_{i}\int_{\Omega^{i}_{n, k}} (A_i(x, u^n_i, \nabla u^n_i) - A_i(x, T_k u^n_i, 0)) dx\\
& \quad +\int_{\Omega}\left(G(x, \bu_n) - G(x, \T_k\bu_n)\right) dx,
\end{split}
\]
where \eqref{exh_2}, \eqref{gezero}, \eqref{aizero}, respectively 
the boundness of $A_i(x, T_k u^n_i, 0)$ and \eqref{limmeas0}, imply that
\[
\int_{\Omega^{i}_{n, k}} A_i(x, u^n_i, \nabla u^n_i)\ dx\longrightarrow 0
\quad\hbox{and}\quad  \int_{\Omega^{i}_{n, k}} A_i(x, T_k u^n_i, 0)\, dx\longrightarrow 0
\]
for all $i \in \{1,\dots,m\}$. On the other hand, from \eqref{Gsigmamax}, \eqref{non_so}, \eqref{3.17}, \eqref{3.18}, \eqref{Tkae} 
and, again,  Dominated Convergence Theorem, we have that
\[
\int_{\Omega}\left(G(x, \bu_n) - G(x, \T_k\bu_n)\right) dx\longrightarrow 0.
\]
Thus, \eqref{step3.2} follows from \eqref{3.15}.\\
{\sl Step 4.} For each $i\in\{1,\dots, m\}$, by using the same arguments as in {\sl Step 4}
 in the proof of \cite[Proposition 4.6]{CP2} but applied to the partial derivative
$\frac{\partial \J}{\partial u_i}(\T_k\bu_n)$, we prove that
\begin{equation}   \label{TkWi}
\Vert T_k u^n_i - u_i\Vert_{W_i}\to 0.
\end{equation}
So, condition $(i)$ follows from \eqref{Rk_0} and \eqref{TkWi}.\\
{\sl Step 5.} By applying Proposition \ref{smooth1} to the uniformly bounded sequence $(\T_k\bu_n)_n$, 
 from $(i)$ and \eqref{Tkae} we have that
\[
\J(\T_k\bu_n)\to \J(\bu) \quad\mbox{ and }\quad\Vert d\J(\T_k\bu_n) -d\J(\bu)\Vert_{X^{\prime}}\to 0,
\]
which, together with \eqref{step3.1} and \eqref{step3.2}, implies $(ii)$. 
\end{proof}


\section{Main results}    \label{sec_main}
In order to prove a multiplicity result, for each $i \in \{1,\dots,m\}$ we make use of the decomposition 
of $W_i$ as introduced in \cite[Section 5]{CP2}. 
More precisely, it is known that the first eigenvalue of
$- \Delta_{p_i}$ in $W_i$ is given by
\begin{equation}\label{autoval_i}
\lambda_{i,1}:=\inf_{u\in W_i\setminus\lbrace 0\rbrace}
\frac{\int_{\Omega}\vert\nabla u\vert^{p_i}dx}{\int_{\Omega}\vert u\vert^{p_i} dx},
\end{equation}
it is simple, strictly positive and isolated and has a unique eigenfunction $\varphi_{i,1}$
such that
\begin{equation}\label{eig}
\varphi_{i,1} > 0 \;\hbox{a.e. in $\Omega$,}\quad \varphi_{i,1} \in L^\infty(\Omega)
\quad\hbox{and}\quad |\varphi_{i,1}|_{p_i}=1
\end{equation}
(see, e.g., \cite{Lin}). Then, a sequence of positive real numbers $(\lambda_{i,n})_n$
and a corresponding sequence of pseudo--eigenfunctions $(\psi_{i,n})_n \subset W_i$
 exist such that
\begin{equation}    \label{lambda_m}
0 <\lambda_{i, 1} <\lambda_{i, 2}\leq\dots\leq\lambda_{i,n}\leq\dots, \quad 
\mbox{ with } \lambda_{i,n}\nearrow +\infty \ \mbox{ as } n\to +\infty,
\end{equation}
and for all $n \in \N$ we have that $\psi_{i,n} \in L^{\infty}(\Omega)$, hence 
$\psi_{i,n} \in X_i$. Moreover, taking $k \in \N$ and
\[
V_{i, k} := {\rm span}\lbrace\psi_{i, 1},\dots,\psi_{i, k}\rbrace,
\]
a suitable (infinite dimensional) topological complement $Y_{i,k}$ can be found in $W_i$ such that 
\[
W_i= V_{i,k}\oplus Y_{i,k}
\]
and the following inequality holds:
\begin{equation}\label{lambdan+1}
\lambda_{i,k+1}\ \int_\Omega|w|^{p_i} dx\ \leq\ \int_\Omega |\nabla w|^{p_i} dx \quad
\hbox{ for all }  w\in Y_{i,k}
\end{equation}
(cf. \cite[Proposition 5.4]{CP2}).
Hence, defining $Y_{k}^{X_i} := Y_{i, k}\cap L^{\infty}(\Omega)\subset X_i$, from the boundedness
of each $\psi_{i,n}$ we have that $V_{i, k}$ is a subspace of $X_i$, too,
and then
\[
X_i= V_{i, k}\oplus Y_k^{X_i},
\]
with
\[
\mbox{dim}V_{i, k} = \mbox{codim}Y_{k}^{X_i} = k.
\]
Thus, by means of \eqref{Wdefn1} we have that
\[
W = (V_{1, k}\times\dots\times V_{m, k})\oplus(Y_{1, k}\times\dots\times Y_{m, k})
\]
and then from \eqref{Xdefn1} it follows that
\[
X= V_k \oplus Y_{k}^{X}
\]
with
\begin{equation}\label{subsp}
\begin{split}
&V_k = V_{1, k}\times\dots\times V_{m, k}, \qquad Y^X_k = Y_{k}^{X_1}\times\dots\times Y_{k}^{X_m},\\
&\mbox{where } {\rm dim} V_k = {\rm codim} Y^X_k <+\infty.
\end{split}
\end{equation}

Now, we are able to state our main results.

\begin{theorem}    \label{ThExist}
For each $i\in\{1,\dots, m\}$ let $p_i >1$ and assume
that $A_i(x, t, \xi)$ satisfies hypotheses $(h_0)$--$(h_6)$. 
Moreover, suppose that a given function $G(x, \bu)$ verifies
$(g_0)$--$(g_2)$. If, in addition, $\mu_3>0$ exists such that
\begin{enumerate}
\item[$(h_7)$] for every $i\in\{1,\dots, m\}$ it is
\[
A_i(x, t, \xi)\geq\mu_3\vert\xi\vert^{p_i} \quad \mbox{ a.e. in } \Omega, 
\ \mbox{ for all } (t, \xi)\in\R\times\R^N;
\]
\item [$(g_3)$] taking $\lambda^*_{1}:= \min\big\lbrace \lambda_{i,1}:\ 1 \le i \le m\big\rbrace$, 
with $\lambda_{i,1}$ as in \eqref{autoval_i}, assume that
\[
\limsup_{\bu\to {\bf 0}} \frac{G(x, \bu)}{\sum_{i} \vert u_i\vert^{p_i}} < 
\mu_3 \lambda^*_{1} \quad \hbox{uniformly a.e. in $\Omega$;}
\]
\end{enumerate}
then functional $\J$ in \eqref{functional} possesses at least one 
nontrivial critical point in $X$; hence, problem \eqref{euler} 
admits a nontrivial weak bounded solution.
\end{theorem}

\begin{theorem}     \label{ThMolt}
For each $i\in\{1,\dots, m\}$ let $p_i >1$ and
assume that $A_i(x,t,\xi)$ and $G(x, \bu)$ satisfy hypotheses $(h_0)$--$(h_6)$, $(g_0)$--$(g_2)$. 
If we also suppose that
\begin{enumerate}
\item[$(h_8)$] for each $i\in\{1,\dots, m\}$ function $A_i(x, \cdot, \cdot)$ 
is even in $\R\times\R^N$ for a.e. $x\in\Omega$;
\item[$(g_4)$] $\displaystyle\liminf_{|\bu|\to +\infty}\frac{G(x, \bu)}{\sum_{i} \vert u_i\vert^{\frac{1}{\theta_i}}} >0
\quad$ uniformly a.e. in $\Omega$;
\item[$(g_5)$] $G(x,\cdot)$ is even in $\R^m$ for a.e. $x\in\Omega$;
\end{enumerate}
then functional $\J$ in \eqref{functional} has an unbounded sequence of critical 
points $(\bu_n)_n\subset X$ such that $\J(\bu_n)\nearrow +\infty$; 
hence, problem \eqref{euler} admits infinitely many distinct weak bounded solutions.
\end{theorem}

From now on, assume that for each $i \in \{1,\dots,m\}$ 
function $A_i(x,t,\xi)$ satisfies $(h_0)$--$(h_6)$
while $G(x, \bu)$ verifies $(g_0)$--$(g_2)$.
Moreover, let $1<p_i< N$ for each $i\in\{1,\dots, m\}$ (otherwise,
the proof is simpler) and, for simplicity, suppose that
\begin{equation}\label{zeroi}
\int_{\Omega}A_i(x,0,\textbf{0}_N) dx = 0\quad
\hbox{for all $i \in \{1,\dots,m\}$,}
\end{equation}
with $\textbf{0}_N = (0,\dots,0) \in \R^N$, and 
\begin{equation}\label{zeroG}
\int_{\Omega}G(x,\textbf{0}) dx = 0
\end{equation}
(otherwise, we replace $\J(\bu)$ with 
$\J(\bu) - \sum_i \int_{\Omega}A_i(x,0,\textbf{0}_N) dx + \int_{\Omega}G(x,\textbf{0}) dx$ 
which has the same differential on $X$).

Firstly, we state the following preliminary result (for the proof, see \cite[Proposition 6.5]{CP2}).

\begin{proposition}
Let $i\in\{1,\dots, m\}$ be fixed.
Then, for any $(t, \xi)\in\R\times\R^N$ with $\vert(t, \xi)\vert\geq R$ we have that
\[
A_i(x, st, s\xi)\leq s^{\frac{1}{\theta_i}\left(1-\frac{\mu_2}{\eta_1}\right)} A_i(x, t, \xi)
\quad \hbox{a.e. in $\Omega$, for all $s\geq 1$,}
\]
with $R$, $\theta_i$, $\mu_2$ and $\eta_1$ as in our set of hypotheses. 
Moreover, some constants $b^*_1$, $b_2^* >0$ 
exist, independent of $i$, such that for all $(t,\xi)\in\R\times\R^N$ it is
\begin{equation}     \label{stsec}
\vert A_i(x, t,\xi)\vert\leq b^*_1\left(1 +\vert t\vert^{\frac{1}{\theta_i}\left(1-\frac{\mu_2}{\eta_1}\right)}\right) 
+ b^*_2 \left(1 +\vert t\vert^{\frac{1}{\theta_i}\left(1-\frac{\mu_2}{\eta_1}\right) - p_i}\right)\vert\xi\vert^{p_i} \ 
\end{equation}
a.e. in $\Omega$, with $\frac{1}{\theta_i}\left(1-\frac{\mu_2}{\eta_1}\right) - p_i >0$ (without loss of generality,
as we can take, a priori, either $\mu_2$ small enough or $\eta_1$ large enough).
\end{proposition}

Now, we are able to prove our main results. 

\begin{proof}[of Theorem \ref{ThExist}]
We note that, being $\mu_3 \lambda^*_1 > 0$, then from hypothesis $(g_3)$ 
a constant $\overline{\lambda} \in \R$ exists such that
\begin{equation}   \label{lambdasegnato}
\overline{\lambda} >0 \quad \hbox{and}\quad
\limsup_{\bu\to {\bf 0}} \frac{G(x, \bu)}{\sum_{i} \vert u_i\vert^{p_i}} <
\overline{\lambda}< \mu_3 \lambda^*_1 \quad \hbox{uniformly a.e. in $\Omega$.}
\end{equation}
Thus, a radius $\rho^{\ast} >0$ exists such that
\begin{equation}   \label{Glam}
G(x, \bu)\ \leq\ \overline{\lambda} \sum_{i} \vert u_i\vert^{p_i} 
\quad \mbox{ for a.e. $x\in\Omega\ $ if }\ \vert\bu\vert\leq\rho^{\ast}.
\end{equation}
Now, let $\bu \in \R^m$ be such that $\vert\bu\vert>\rho^{\ast}$,
$\bu = (u_1,\dots,u_m)$.
Then, an integer $i^*\in\{1,\dots, m\}$ 
exists such that $\vert u_{i^*}\vert\geq\frac{\rho^{\ast}}{m}$;
hence,  
$1 \le \left(\frac{m}{\rho^{\ast}}\right)^{\overline{q}_{i^*}}\vert u_{i^*}\vert^{\overline{q}_{i^*}}$. 
Thus, from \eqref{Gsigmamax} 
direct computations allow us to prove the existence of
a constant $\sigma^{\ast} > 0$ such that
\begin{equation}   \label{Glamstar}
|G(x, \bu)| \leq \sigma^{\ast}\sum_{i}\vert u_i\vert^{\overline{q}_i}
\quad \mbox{ for a.e. $x\in\Omega\ $ if }\ \vert\bu\vert > \rho^{\ast}.
\end{equation}
Summing up, from \eqref{Glam} and \eqref{Glamstar}, we have that 
\begin{equation}   \label{Gdis}
G(x, \bu) \leq\overline{\lambda} \ \sum_{i} \vert u_i\vert^{p_i} 
+ \sigma^{\ast}\sum_{i}\vert u_i\vert^{\overline{q}_i}
\quad \mbox{ for a.e. $x\in\Omega\ $, all $\bu \in \R^m$. }
\end{equation}
Then, taking $\bu \in X$,
from \eqref{functional}, assumption $(h_7)$ and inequality \eqref{Gdis} 
we have that
\[
\J(\bu)\geq \sum_{i}\left( \mu_3\int_{\Omega}\vert\nabla u_i\vert^{p_i}dx
 -\overline{\lambda}\int_{\Omega}\vert u_i\vert^{p_i} dx
- \sigma^{\ast} \int_{\Omega}\vert u_i\vert^{\overline{q}_i}dx \right),
\]
or better, \eqref{Sobpi}, \eqref{non_so}, 
\eqref{autoval_i}, the definition of $\lambda^*_1$ and also
\eqref{Wnorm} imply that
\begin{equation}    \label{5.6conto}
\begin{split}
\J(\bu) \ &\geq \ \sum_{i}\left(\Big(\mu_3 - \frac{\overline{\lambda}}{\lambda_{i,1}}\Big) \Vert u_i\Vert_{W_i}^{p_i} -
\bar{\sigma} \Vert u_i\Vert_{W_i}^{\overline{q}_i} \right)\\
&\geq\sum_{i} \left(\Vert u_i\Vert_{W_i}^{p_i}\Big(\mu_3 -\frac{\overline{\lambda}}{\lambda^*_{1}}
- \bar{\sigma} \Vert \bu\Vert_{W}^{\overline{q}_i - p_i}\Big) \right)
\end{split}
\end{equation}
for a suitable $\bar\sigma > 0$.
We note that from \eqref{non_so} and \eqref{lambdasegnato} some constants $R_0$, $\rho_1$ exist such that
\begin{equation}    \label{conto3}
0 < R_0 < 2m \qquad \hbox{and}\qquad
\mu_3 -\frac{\overline{\lambda}}{\lambda_{1}^*} -\bar{\sigma} R_0^{\overline{q}_i -p_i} \ge\ \rho_1 >0
\end{equation}
for all $ i\in\{1,\dots, m\}$. Since \eqref{Wnorm} and $\Vert\bu\Vert_W = R_0$ imply that 
$\Vert u_j\Vert_{W_j} \ge \frac{R_0}{2m}$ for some $j \in \{1,\dots,m\}$,
from \eqref{5.6conto} and \eqref{conto3}
we have that
\[
\J(\bu)\geq\rho_1 \sum_{i} \Vert u_i\Vert_{W_i}^{p_i}
\ge \rho_1 \Vert u_j\Vert_{W_j}^{p_j} \ge 
\rho_1 \left(\frac{R_0}{2m}\right)^{\bar{p}}
\]
with $\bar{p} = \max\{p_i:\ 1\le i\le m\}$.
So, a constant $\rho_0 > 0$ can be found so that
\begin{equation}    \label{geqrho}
\bu \in X, \quad\Vert\bu\Vert_W = R_0\qquad\then\qquad
\J(\bu)\geq\rho_0.
\end{equation}
Finally, in order to prove that also the geometric condition
$(ii)$ in Theorem \ref{mountainpass} is satisfied,
fixing any $i\in\{1,\dots, m\}$, we take $\varphi_{i, 1}\in X_i$ as in \eqref{eig}.
Then, from \eqref{functional}, \eqref{Gthetaibis}, \eqref{zeroi}, \eqref{stsec}
 and direct computations, for any $s>0$ we obtain that
\[
\begin{split}
&\J(s \varphi_{i,1}{\bf e}_i) \ =
\int_\Omega A_i(x,s \varphi_{i,1},s \nabla\varphi_{i,1}) dx +
\sum_{j\ne i} \int_\Omega A_j(x,0,{\bf 0}_N) dx\\&\quad\qquad - \int_\Omega G(x,s \varphi_{i,1}\textbf{e}_i) dx\\
&\qquad\leq\ b_1^* \meas(\Omega) + b_1^* s^{\frac{1}{\theta_i}(1-\frac{\mu_2}{\eta_1})}
\int_{\Omega} |\varphi_{i,1}|^{\frac{1}{\theta_i}(1-\frac{\mu_2}{\eta_1})} dx+ b_2^* s^{p_i} \int_{\Omega} |\nabla\varphi_{i,1}|^{p_i} dx \\
&\quad\qquad+b_2^* s^{\frac{1}{\theta_i}(1-\frac{\mu_2}{\eta_1})} 
\int_{\Omega} |\varphi_{i,1}|^{\frac{1}{\theta_i}(1-\frac{\mu_2}{\eta_1}) - p_i}|\nabla\varphi_{i,1}|^{p_i} dx\\
&\quad\qquad -s^{\frac{1}{\theta_i}}\int_{\Omega} h_i(x) |\varphi_{i,1}|^{\frac{1}{\theta_i}} dx
+ \sigma_i \meas(\Omega),
\end{split}
\]
where \eqref{eig} and Remark \ref{RemarkG} imply that all the integrals are finite
and 
\[
\int_{\Omega} h_i(x) |\varphi_{i,1}|^{\frac{1}{\theta_i}} dx > 0.
\]
Thus, since $\eta_1, \mu_2 >0$, from \eqref{minimax} it follows that
\[
\J(s \varphi_{i,1}{\bf e}_i) \rightarrow -\infty \quad \mbox{ as } s\rightarrow +\infty.
\]
Hence, $v\in X_i$ exists such that
\begin{equation}    \label{ei}
\Vert v {\bf e}_i\Vert_W > R_0\quad\mbox{ and }\quad\J(v {\bf e}_i) < \varrho_0,
\end{equation}
$R_0$ and $\varrho_0$ as in \eqref{geqrho}.\\
So, since \eqref{functional}, \eqref{zeroi} and \eqref{zeroG} imply that
$\J(\textbf{0}) = 0$, by means of Theorem \ref{mountainpass} we have that
Propositions \ref{smooth1} and \ref{PropwCPS} together with 
\eqref{geqrho} and \eqref{ei} imply the existence of a critical point $\bu^{\ast}\in X$ 
such that $\J(\bu^{\ast})\geq\varrho_0 > \J(\bf{0})$. 
\end{proof}

\begin{proof}[of Theorem \ref{ThMolt}]
Since $(h_8)$ and $(g_5)$ imply that $\J$ is an even functional on $X$,
in order to apply Corollary \ref{multiple} we have to prove that 
assumption $({\mathcal H}_\varrho)$ holds for infinitely many $\varrho$.\\
For simplicity, here and in the following, $b_l$ will denote any strictly positive constant 
independent of the index $i \in \{1,\dots,m\}$.\\
Firstly, we prove that,
taking any finite dimensional subspace $V$ of $X$, a radius
$R_V > 0$ exists such that
\begin{equation}\label{uno1}
\J(\bu) \le 0\qquad \hbox{for all $\bu \in V$, $\|\bu\|_X \ge R_V$.}
\end{equation}
To this aim, for any $i\in\{1,\dots, m\}$ and $u_i \in X_i$, from \eqref{stsec},
and, then, \eqref{Xinorm}, direct computations imply that
\begin{equation}\label{Ainf}
\begin{split}
&\int_{\Omega} A_i(x, u_i, \nabla u_i) dx \leq b^*_1\meas(\Omega)\left(1 +\vert u_i\vert_{\infty}^{\frac{1}{\theta_i}\left(1-\frac{\mu_2}{\eta_1}\right)}\right)\\
&+b^*_2\left(1 +\vert u_i\vert_{\infty}^{\frac{1}{\theta_i}\left(1-\frac{\mu_2}{\eta_1}\right)-p_i}\right)\Vert u_i\Vert_{W_i}^{p_i}\le b_1 + b_2 \|u_i\|_{X_i}^{\frac{1}{\theta_i}\left(1-\frac{\mu_2}{\eta_1}\right)} 
+ b_2^*\|u_i\|_{X_i}^{p_i}.
\end{split}
\end{equation}
On the other hand, from hypothesis $(g_4)$ it follows that 
$\bar{\bar{\lambda}}$ exists such that
\[
\liminf_{\vert\bu\vert\to +\infty}\frac{G(x, \bu)}{\sum_{i}\vert u_i\vert^{\frac{1}{\theta_i}}}>\bar{\bar{\lambda}} >0 
\quad\mbox{ uniformly a.e. in }\Omega,
\]
so $R_1 >0$ exists so that
\[
G(x,\bu) \geq \bar{\bar{\lambda}}\ \sum_{i} \vert u_i\vert^{\frac{1}{\theta_i}} 
\quad\mbox{ for a.e. $x\in\Omega$, if }\vert\bu\vert\geq R_1,
\]
while from \eqref{Gsigma} we obtain 
\[
\big|G(x,\bu) - \bar{\bar{\lambda}}\ \sum_{i} \vert u_i\vert^{\frac{1}{\theta_i}}\big| \le \bar{\bar{\sigma}} 
\quad\mbox{ for a.e. $x\in\Omega$, if }\vert\bu\vert\le R_1.
\]
Hence, we have that
\begin{equation}\label{uno2}
G(x,\bu) \geq \bar{\bar{\lambda}}\ \sum_{i} \vert u_i\vert^{\frac{1}{\theta_i}} - \bar{\bar{\sigma}} 
\quad\mbox{ for a.e. $x\in\Omega$, for all $\bu\in \R^m$.}
\end{equation}
Thus, from \eqref{functional}, \eqref{Ainf} and \eqref{uno2} we obtain that
\begin{equation}\label{uno3}
\J(\bu)\leq \sum_{i} \left(b_3 + b_2\|u_i\|_{X_i}^{\frac{1}{\theta_i}\left(1-\frac{\mu_2}{\eta_1}\right)} 
+ b_2^* \|u_i\|_{X_i}^{p_i} - \bar{\bar{\lambda}} \int_\Omega \vert u_i\vert^{\frac{1}{\theta_i}}dx\right)
\end{equation}
for all $\bu \in X$. Now, if $V$ is a finite dimensional subspace of $X$,
$V =V_1\times\dots\times V_m$, for any $i \in \{1,\dots,m\}$
also the projection $V_i$ onto $X_i$ is a finite dimensional subspace,
so all the norms are equivalent and in particular $\nu_i >0$ exists such that
\[
\left(\int_\Omega \vert v\vert^{\frac{1}{\theta_i}}dx\right)^{\theta_i}
\ge  \nu_i \|v\|_{X_i} \qquad \hbox{for all $v \in V_i$.} 
\]
So, from \eqref{uno3} it follows that
\begin{equation}\label{uno4}
\J(\bu)\leq \sum_{i} \left(b_3 + b_2\|u_i\|_{X_i}^{\frac{1}{\theta_i}\left(1-\frac{\mu_2}{\eta_1}\right)} 
+ b_2^* \|u_i\|_{X_i}^{p_i} - b_4 \|u_i\|_{X_i}^{\frac{1}{\theta_i}}\right)
\end{equation}
for all $\bu \in V$ or better, taking $i\in\{1,\dots, m\}$ and
\[
\zeta_i: s \in [0,+\infty[ \mapsto \zeta_i(s) = b_3 + b_2 s^{\frac{1}{\theta_i}\left(1-\frac{\mu_2}{\eta_1}\right)} 
+ b_2^* s^{p_i} - b_4 s^{\frac{1}{\theta_i}} \in \R,
\]
estimate \eqref{uno4} reduces to
\begin{equation}\label{uno5}
\J(\bu)\leq \sum_{i} \zeta_i(\|u_i\|_{X_i})
\qquad \hbox{for all $\bu \in V$.}
\end{equation}
We note that for each $i \in \{1,\dots,m\}$ the continuous function $\zeta_i$ 
is such that $\zeta_i(0) = b_3>0$ and
$\zeta_i(s) \to -\infty$ as $s \to+\infty$. Hence, a constant $b_5> 0$ and
a radius $\bar{R} > 0$ exist such that
\begin{equation}\label{uno6}
\max_{s \ge 0} \zeta_i(s) \le b_5\quad 
\hbox{and}\quad \zeta_i(s) < -m b_5 \;\; \hbox{if $\ s \ge \bar{R}$}\quad
\hbox{for all $i \in \{1,\dots,m\}$.}
\end{equation}
Thus, if $(s_1,\dots,s_m) \in (\R_+)^m$ is such that $\sum_i s_i > m\bar{R}$,
we have that $s_j \ge \bar{R}$ for some index $j \in \{1,\dots,m\}$, 
then \eqref{uno6} and direct computations allow one to prove that
\[
\sum_{i} \zeta_i(s_i) < 0 \quad\hbox{if }\;\ \sum_i s_i > m\bar{R}.
\]
Whence, \eqref{uno1} holds from \eqref{uno5} with $R_V \ge m\bar{R}$,
$s_i = \|u_i\|_{X_i}$.\\
Now, we want to prove that for any fixed $\varrho >0$ 
there exists $k=k(\varrho)\geq 1$ and $R_k > 1$ such that
\begin{equation}\label{uno7}
\bu\in Y_k^X, \quad \Vert \bu\Vert_{W} = R_k
\qquad \implies\qquad \J(\bu)\geq \varrho
\end{equation}
with $Y_k^X$ as in \eqref{subsp}.\\
To this aim, we note that if $\bu \in X$ estimate \eqref{4.4} gives
\[
\int_{\Omega} A_i(x, u_i, \nabla u_i) dx\geq
\lambda(\theta_i \mu_1+\mu_2)\int_{\Omega}\vert\nabla u_i\vert^{p_i} dx -\eta_3\meas(\Omega),
\]
for each $i\in\{1,\dots, m\}$. Then, \eqref{functional} and \eqref{Gsigmamax} imply that
\begin{equation}     \label{J5.5}
\begin{split}
\J(\bu) \ &\geq \ \sum_{i}\left(\lambda(\theta_i \mu_1+\mu_2)\int_{\Omega}\vert\nabla u_i\vert^{p_i} dx
- \eta_3\meas(\Omega)\right.\\
&\left.\qquad\qquad- \sigma_0 \int_{\Omega}\vert u_i\vert^{\overline{q}_i}dx - \sigma_0 \meas(\Omega)\right).
\end{split}
\end{equation}
We note that, for every $i\in\{1,\dots, m\}$, inequality \eqref{non_so} 
allows us to take $0 < r_i < p_i$ such that 
\[
\frac{r_i}{p_i}+\frac{\bar{q}_i-r_i}{p_i^{\ast}} =1, \quad \hbox{i.e.}\quad 
r_i = p_i \frac{p^*_i - \bar{q}_i}{p_i^{\ast} - p_i}, 
\]
so that from \eqref{Sobpi}, standard interpolation arguments and,
fixing any $k \in \N$, condition \eqref{lambdan+1}, it follows that
\begin{equation}   \label{interp}
\int_{\Omega}\vert u_i\vert^{\overline{q}_i}dx \leq
\tau^{\bar{q}_i -r_i}_{i,p_i^{\ast}}\left(\frac{1}{\lambda_{i,k+1}}\right)^{\frac{r_i}{p_i}}\|u_i\|_{W_i}^{\bar{q}_i} 
\quad\mbox{ for all } u_i\in Y_k^{X_i}.
\end{equation}
Thus, from \eqref{J5.5} and \eqref{interp}, taking $b_6$, $b_7$, $b_8 > 0$
such that
\[
b_6 = \min_{1\le i\le m} \lambda(\theta_i \mu_1+\mu_2), \quad
b_7 = \max_{1\le i\le m} \tau^{\bar{q}_i -r_i}_{i,p_i^{\ast}} \sigma_0, \quad
b_8 = m (\eta_3 + \sigma_0) \meas(\Omega),
\]
we obtain that
\begin{equation}     \label{uno9}
\J(\bu)\geq\sum_{i} \left(\|u_i\|_{W_i}^{p_i} \left(b_6 - 
\ \frac{b_7}{\bar{\lambda}_k}\ \|u_i\|_{W_i}^{\bar{q}_i - p_i}\right)\right) - b_8
\quad\mbox{ for all } \bu\in Y_k^{X},
\end{equation}
with
\[
\bar{\lambda}_k =\min_{1\le i\le m} (\lambda_{i, k+1})^{\frac{r_i}{p_i}}.
\]
We note that \eqref{lambda_m} implies 
\begin{equation}   \label{maxmin}
\bar{\lambda}_k \to +\infty\quad \hbox{as}\; k\to +\infty,
\end{equation}
then, if we assume 
\[
R_k = \left(\frac{b_6}{2b_7} \bar{\lambda}_k\right)^{\frac{1}{\bar{q} - p}}, 
\]
with
\[
p =\min_{1\le i\le m} p_i,\qquad \bar{q} =\max_{1\le i\le m} \bar{q}_i,
\]
$\bar{q} > p > 1$ from \eqref{non_so}, limit \eqref{maxmin} gives
\begin{equation}   \label{maxmin1}
R_k \to + \infty\quad \hbox{as}\; k\to +\infty.
\end{equation}
Thus, an integer $k_1 \in \N$ exists such that
for all $k\ge k_1$ it is $R_k \ge 1$
and 
\[
R_k^{\bar{q}_i - p_i} \le R_k^{\bar{q} - p}
\quad \hbox{for all $i\in\{1,\dots, m\}$.}
\]
Or better, if $k_2\ge k_1$ is such that for all $k\ge k_2$
we have $R_k \ge 2 m$,
taking $k\ge k_2$ and $\bu\in Y_k^{X}$ such that
$\|\bu\|_W = R_k$, direct computations imply not only that
\[
b_6 - \ \frac{b_7}{\bar{\lambda}_k}\ \|u_i\|_{W_i}^{\bar{q}_i - p_i}
\ge b_6 - 
\ \frac{b_7}{\bar{\lambda}_k}\ R_k^{\bar{q}_i - p_i} \ge \frac{b_6}{2}
\quad \hbox{for all $i\in\{1,\dots, m\}$,}
\]
but also
\[
\sum_{i} \|u_i\|_{W_i}^{p_i} \ge \left(\frac{R_k}{2 m}\right)^{p}.
\]
Hence, \eqref{uno9} gives 
\[
\J(\bu)\geq \frac{b_6}{2} \left(\frac{R_k}{2 m}\right)^{p} - b_8
\quad \hbox{if $\bu \in Y_k^X$ is such that $\|\bu\|_W = R_k$.}
\]
Thus, taking any $\varrho >0$, \eqref{uno7} follows from \eqref{maxmin1} if $k = k(\varrho) \in \N$ 
is large enough.\\
Moreover,  taking $\bar{k}> k$, 
the $\bar{k}$--dimensional space $V_{\bar{k}}$, defined as in \eqref{subsp},
is such that not only ${\rm codim} Y_{k} < {\rm dim} V_{\bar k}$
but also \eqref{uno1} holds.
Whence, assumption $({\mathcal H}_\varrho)$ is verified
with $\M_\varrho = \{\bu \in X :\ \Vert\bu\Vert_{W} = R_k\}$.\\
Finally, since \eqref{zeroi} and \eqref{zeroG} give
$\J({\bf 0}) = 0$, from Propositions \ref{smooth1},
\ref{PropwCPS} and the arbitrariness of $\varrho$ for condition
$({\mathcal H}_\varrho)$, we have that Corollary \ref{multiple} applies to $\J$ in $X$ 
and a sequence of diverging critical levels exists. 
\end{proof}



\begin{thebibliography}{99}

\bibitem{AR} 
A. Ambrosetti and P.H. Rabinowitz, 
Dual variational methods in critical point theory and applications, 
\emph{J. Funct. Anal.} {\bf 14} (1973), 349-381.

\bibitem{AG} G. Arioli and F. Gazzola, Existence and multiplicity results for quasilinear elliptic
differential systems,
\emph{Comm. Partial Differential Equations} {\bf 25} (2000), 125-153. DOI:
10.1080/03605300008821510

\bibitem{BBF}
P. Bartolo, V. Benci and D. Fortunato,
{Abstract critical point theorems and applications 
to some nonlinear problems with ``strong'' resonance at infinity},
\emph{Nonlinear Anal.} {\bf 7} (1983), 981-1012.

\bibitem{dF1} L. Boccardo and G. de Figueiredo,
Some remarks on a system of quasilinear elliptic equations,
\emph{NoDEA Nonlinear Differential Equations Appl.} {\bf 9} (2002), 309–323.

\bibitem{Br}
H. Brezis,
\emph{Functional Analysis, Sobolev Spaces 
and Partial Differential Equations}, Universitext Springer, New York, 2011. 

\bibitem{CMPP}
A.M. Candela, E. Medeiros, G. Palmieri and K. Perera,
Weak solutions of quasilinear elliptic systems via the cohomological index,
\emph{Topol. Methods Nonlinear Anal.} {\bf 36} (2010), 1--18.

\bibitem{CP1} A.M. Candela and G. Palmieri, {Multiple solutions
of some nonlinear variational problems}, \emph{Adv. Nonlinear Stud.} {\bf 6} (2006), 269-286.

\bibitem{CP2}
A.M. Candela and G. Palmieri,
{Infinitely many solutions of some nonlinear variational equations},
\emph{Calc. Var. Partial Differential Equations} {\bf 34} (2009), 495-530.

\bibitem{CP3} A.M. Candela and G. Palmieri, {Some abstract critical point theorems
and applications}. In: \emph{Dynamical Systems, Differential Equations and Applications} 
(X. Hou, X. Lu, A. Miranville, J. Su \& J. Zhu Eds), 
\emph{Discrete Contin. Dynam. Syst.} \textbf{Suppl. 2009} (2009), 133-142. 

\bibitem{CP2017}
A.M. Candela and G. Palmieri,
Multiplicity results for some nonlinear elliptic problems
with asymptotically $p$--linear terms,
\emph{Calc. Var. Partial Differential Equations} \textbf{56}:72 (2017).

\bibitem{CPS_CCM}
A.M. Candela, G. Palmieri and A. Salvatore,
Multiple solutions for some symmetric supercritical problems, 
\emph{Commun. Contemp. Math.} \textbf{22} (2020), Article 1950075 (20 pages).
DOI: 10.1142/S0219199719500755

\bibitem{CSS}
A. M. Candela, A. Salvatore and C. Sportelli,
Existence and multiplicity results for a class of coupled quasilinear elliptic systems of gradient type,
\emph{Adv. Nonlinear Stud.} (to appear).

\bibitem{dF2} D. de Figueiredo, Nonlinear elliptic systems,
\emph{An. Acad. Brasil. Ciênc.} {\bf 72}, (2000), 453–469.

\bibitem{dFR} D. de Figueiredo, J. M. do Ó and B. Ruf, 
Non--variational elliptic systems in dimension two: a priori bounds and existence of positive solutions,
\emph{J. Fixed Point Theory App.l} {\bf 4} (2008), 77–96.

\bibitem{FT} L. F.O. Faria, O. H. Miyagaki, D. Motreanu and M. Tanaka, 
Existence results for nonlinear elliptic equations with Leray--Lions operator and dependence on the gradient,
\emph{Nonlinear Anal.} {\bf 96} (2014), 154–166.

\bibitem{LU} O.A. Ladyzhenskaya and N.N. Ural'tseva, \emph{Linear and Quasilinear Elliptic
Equations}, Academic Press, New York, 1968.

\bibitem{Lin} P. Lindqvist, On the equation 
${\rm div} (|\nabla u|^{p-2}\nabla u) + \lambda |u|^{p-2}u =0$, 
\emph{Proc. Amer. Math. Soc.} {\bf 109} (1990), 157-164.

\bibitem{PeSq} B. Pellacci and M. Squassina, Unbounded critical points for a class
of lower semicontinuous functionals, \emph{J. Differential Equations} {\bf 201}
(2004), 25-62.

\bibitem{PAO}
K. Perera, R.P. Agarwal and D. O'Regan,
\emph{Morse Theoretic Aspects of $p$--Laplacian Type Operators},
Math. Surveys Monogr. {\bf 161}, Amer. Math. Soc., Providence RI, 2010.

\bibitem{Ra1}
P.H. Rabinowitz,
\emph{Minimax Methods in Critical Point Theory with Applications to Differential 
Equations}, CBMS Regional Conf. Ser. in Math. {\bf 65}, Amer. Math. Soc., Providence, 1986.

\bibitem{VT} J. Vélin and F. de Thélin, Existence and nonexistence of nontrivial solutions for
some nonlinear elliptic systems,
\emph{Rev. Mat. Univ. Complut. Madrid} {\bf 6} (1993), 153–194.
\end{thebibliography}
\end{document}